\RequirePackage{fix-cm}
\documentclass[smallextended,envcountsame,envcountsect]{svjour3}

\smartqed
\journalname{}
\renewcommand{\makeheadbox}{}
\usepackage{amsfonts,amssymb,mathtools}
\usepackage{microtype}
\usepackage{ifpdf}
\usepackage{cite}
\usepackage[hidelinks]{hyperref}

\renewenvironment{proof}[1][Proof]
  {\par\normalfont\trivlist\item[\hskip\labelsep\itshape #1.]\ignorespaces}
  {\qed\endtrivlist}
\AddToHook{env/thebibliography/begin}{\interlinepenalty=10000}

\newcommand{\be}{\begin{equation}}
\newcommand{\ee}{\end{equation}}
\newcommand{\bes}{\begin{equation*}}
\newcommand{\ees}{\end{equation*}}
\newcommand{\beqn}{\begin{eqnarray}}
\newcommand{\eeqn}{\end{eqnarray}}
\newcommand{\beqns}{\begin{eqnarray*}}
\newcommand{\eeqns}{\end{eqnarray*}}

\newcommand{\diag}{\mbox{diag}}

\newcommand{\rank}{\mathrm{rank}}

\newcommand{\bDelta}{\mathbf{\Delta}}

\newcommand{\bp}{\mathbf{p}}
\newcommand{\bq}{\mathbf{q}}

\newcommand{\bu}{\mathbf{u}}
\newcommand{\bv}{\mathbf{v}}

\newcommand{\bA}{\mathbf{A}}
\newcommand{\bB}{\mathbf{B}}
\newcommand{\bC}{\mathbf{C}}

\newcommand{\bE}{\mathbf{E}}

\newcommand{\bG}{\mathbf{G}}

\newcommand{\bI}{\mathbf{I}}
\newcommand{\bJ}{\mathbf{J}}

\newcommand{\bP}{\mathbf{P}}
\newcommand{\bQ}{\mathbf{Q}}
\newcommand{\bR}{\mathbf{R}}

\newcommand{\bU}{\mathbf{U}}
\newcommand{\bV}{\mathbf{V}}
\newcommand{\bW}{\mathbf{W}}
\newcommand{\bX}{\mathbf{X}}
\newcommand{\bY}{\mathbf{Y}}
\newcommand{\bZ}{\mathbf{Z}}

\newcommand{\calM}{{\mathcal M}}

\newcommand{\reals}{\mathbb{R}}

\def\reals{\mathbb{R}}
 \def\b0{\mathbf{0}}
\def\bSigma{\boldsymbol\Sigma}

\def\bDelta{\boldsymbol\Delta}

\def\bDelta{\boldsymbol{\Delta}}

\def\calM{\mathcal{M}}

\def\bI{\mathbf{I}}

\def\reals{\mathbb{R}}
\def\Pr{\mathcal{P}_r}

\newcommand{\argmin}{\operatorname*{arg\; min}}

\providecommand{\bR}{\mathbf{R}}
\providecommand{\bA}{\mathbf{A}}
\providecommand{\bB}{\mathbf{B}}
\providecommand{\bP}{\mathbf{P}}
\providecommand{\bY}{\mathbf{Y}}
\providecommand{\bN}{\mathbf{N}}\providecommand{\bQ}{\mathbf{Q}}
\providecommand{\bJ}{\mathbf{J}}

\numberwithin{equation}{section}

\titlerunning{Critical-Point Characterization and Rank-Constrained ProjGD}
\authorrunning{Teng Zhang}
\title{Critical-Point Characterization of Rank-Constrained Optimization and Global Convergence of Projected Gradient Descent}
\author{Teng Zhang}
\institute{Teng Zhang \at
Dr. Yvette Kanouff School of Data, Mathematical, and Statistical Sciences, University of Central Florida,
Orlando, FL 32816, USA\quad Email: \href{mailto:teng.zhang@ucf.edu}{teng.zhang@ucf.edu}}
\date{}

\ifpdf
\hypersetup{
  pdftitle={Critical-Point Characterization of Rank-Constrained Optimization and Global Convergence of Projected Gradient Descent},
  pdfauthor={Teng Zhang}
}
\fi

\begin{document}
\maketitle

\begin{abstract}
We study the rank-constrained problem
$\min_{\rank(\bX)\leq r} f(\bX)$ when the unconstrained minimizer
$\bX_*$ has rank $r_*\leq r$. We apply a sharp matrix
inequality to characterize
critical points and establish deterministic global linear convergence of
projected gradient descent (ProjGD). First, assuming that $f$ satisfies rank-$(r+r_*)$ restricted smoothness and
strong convexity with constants $L$ and $\mu$, respectively, where
$0<\mu<L$, we show that every nonoptimal rank-$r$ critical point satisfies the inequality 
$\|P_{T(\bX)^\perp}\nabla f(\bX)\|/\sigma_r(\bX)
\geq (2L\mu/(L-\mu))\sqrt{r/r_*}$.
Here $P_{T(\bX)^\perp}\nabla f(\bX)$ is the component of the gradient
normal to the tangent space $T(\bX)$ and $\sigma_r(\bX)$ is the $r$-th singular value of $\bX$.
Second, this inequality implies that $\bX_*$ is the unique
fixed point of ProjGD whenever the step size $\eta$ satisfies 
$\eta>\sqrt{r_*/r}(L-\mu)/(2L\mu)$.
Finally, under an additional rank-$2r$ restricted smoothness assumption,
we apply the argument to the actual ProjGD iterates,
establish a uniform improvement estimate for every iteration, and prove
 \emph{deterministic global linear convergence} from every
initialization of rank at most $r$ for
$\sqrt{r_*/r}(L-\mu)/(2L\mu)<\eta<1/L$.
The result allows arbitrary rank overparameterization $r\geq r_*$,
and its convergence rate is independent of the singular-value
condition number of $\bX_*$.
As a consequence, a rank-preserving curve argument proves
the no-spurious-local-minima property whenever
$(L-\mu)/(L+\mu)<(1+\sqrt{r_*/r})^{-1}$.
This matches the sharp threshold established in \cite{mcrae2026matrixlasso}.
A quadratic example show that the
lower bound on the step-size  $\sqrt{r_*/r}(L-\mu)/(2L\mu)$ is also sharp.%


\end{abstract}

\section{Introduction}\label{sec:introduction}
Low-rank estimation problems in inverse problems, signal processing, and
machine learning frequently take the form
\begin{equation}\label{eq:problem}
\min_{\bX\in\reals^{n\times n}:\,\rank(\bX)\leq r} f(\bX).
\end{equation}
We assume that $f$ is twice continuously differentiable and has an
unconstrained minimizer $\bX_*$ of rank $r_*$, with
$1\leq r_*\leq r<n$. The matrix variable need not be symmetric or
positive semidefinite; the square formulation is only for simplicity.
The search rank $r$ may exceed the intrinsic rank $r_*$.

Two standard approaches to \eqref{eq:problem} are Burer--Monteiro (BM)
factorization $\bX=\bU\bV^T$, followed by optimization in the factors
\cite{Burer:2003ve,pmlr-v54-park17a,zheng2016convergence},
and projected gradient descent (ProjGD), which takes a Euclidean gradient
step and truncates the result to rank $r$
\cite{NIPS2010_08d98638,Zhang2021GeneralLM}.
In the matrix-recovery literature, this same gradient-and-truncation
method is also known as singular-value projection (SVP)
\cite{NIPS2010_08d98638} or iterative hard thresholding (IHT)
\cite{TannerWei2013}.
Ha, Liu, and Barber~\cite{doi:10.1137/18M1231675} identified a
fundamental connection: under an appropriate local step-size bound,
second-order stationary points of the factored objective yield ProjGD
fixed points. Their work includes results for overestimated search rank,
but its qualitative correspondence alone does not ensure that the
\emph{iterates} of ProjGD converge at a uniform linear rate.
Here we address this question through a quantitative critical-point
characterization with an explicit dependence on the rank ratio $r/r_*$.

\subsection{Main contributions: critical-point characterization and ProjGD convergence}
Our contribution has three components. First, we obtain a sharp characterization of nonoptimal rank-$r$ critical points
in terms of their normal gradient and the $r$-th singular value. Second, this characterization gives explicit conditions on rank, condition number, and step size under which $\bX_*$ is the unique ProjGD
fixed point under Assumptions A1--A2.
Third, we extend the argument to a uniform improvement 
estimate for the actual ProjGD iterates. This estimate yields deterministic global linear convergence
from every initialization of rank at most $r$ in the stated step-size
range. 

\textbf{Contribution 1: a matrix inequality and critical-point characterization.} 
Write $\mathcal M_r=\{\bX:\rank(\bX)\leq r\}$,
and let $T(\bX)$ denote the tangent space to the fixed-rank stratum
at a rank-$r$ matrix $\bX$. %
For a nonoptimal rank-$r$ critical point $\bX\ne\bX_*$ of
\eqref{eq:problem}, $P_{T(\bX)}\nabla f(\bX)=0$.
 Our critical-point characterization is a rank-dependent lower
bound on the component of the gradient normal to the tangent space $T(\bX)$:
\begin{equation}\label{eq:intro_certificate}
{\quad
\frac{\|P_{T(\bX)^\perp}\nabla f(\bX)\|}{\sigma_r(\bX)}
\geq \frac{2L\mu}{L-\mu}\sqrt{\frac r{r_*}}.
\quad}
\end{equation}
This holds whenever the objective has the required
rank-$(r+r_*)$ restricted smoothness and strong convexity  parameters $0<\mu<L$. 
\par
The main matrix-inequality ingredient is
Lemma~\ref{lemma:orthogonalrate0}. It shows that a perturbation $\bZ$
normal to the fixed-rank tangent space and satisfying
$\|\bZ+\bX-\bX_*\|_F\leq\kappa_0\|\bX-\bX_*\|_F$,
where $0<\kappa_0<1$, must obey
\[
\|\bZ\|\geq
\sqrt{\frac r{r_*}}\frac{1-\kappa_0^2}{\kappa_0}\sigma_r(\bX).
\]
Lemma~\ref{lemma:orthogonalrate0} has a short, self-contained proof:
completing the square reduces the claim to the Eckart--Young theorem
applied to orthogonal matrix blocks. Combining the inequality with the
restricted centered-gradient estimate of Lemma~\ref{lemma:gradient}
gives \eqref{eq:intro_certificate}. The same matrix argument is used
in the finite-step analysis of Section~\ref{sec:algorithms}, and the
factor $\sqrt{r/r_*}$ quantifies the gain from rank overparameterization.

\textbf{Contribution 2: rank-dependent exclusion of nonoptimal ProjGD fixed points.}
The ratio \eqref{eq:intro_certificate}, together with the step size,
determines whether $\bX$ can be a ProjGD fixed point.
Indeed, for a rank-$r$ point and step size $\eta$, the ProjGD
fixed-point criterion is
\begin{equation}\label{eq:intro_fixedpoint}
P_{T(\bX)}\nabla f(\bX)=0,
\qquad \eta\|\nabla f(\bX)\|\leq\sigma_r(\bX).
\end{equation}
Combining this criterion with \eqref{eq:intro_certificate}
excludes nonoptimal rank-$r$ fixed points whenever
\[
\eta>\eta_{\mathrm{crit}},
\qquad
\eta_{\mathrm{crit}}:=
\sqrt{\frac{r_*}{r}}\frac{L-\mu}{2L\mu}.
\]
Thus, under A1--A2, Theorem~\ref{thm:fixedpoint} shows that
$\bX_*$ is the unique ProjGD fixed point for
$\eta>\eta_{\mathrm{crit}}$.

Combining this lower threshold with the sufficient-decrease
range under A1$'$ gives the step-size interval used for convergence:
\begin{equation}\label{eq:intro_stepsize}
\eta_{\mathrm{crit}}<\eta<\frac1L.
\end{equation}
This interval is nonempty when
\begin{equation}\label{eq:intro_threshold}
\frac{L}{\mu}<1+2\sqrt{\frac r{r_*}},
\quad\text{equivalently}\quad
\delta:=\frac{L-\mu}{L+\mu}
<\frac{\sqrt{r/r_*}}{1+\sqrt{r/r_*}}.
\end{equation}
Thus the same rank-dependent quantity controls both exclusion
of nonoptimal fixed points and the existence of the globally
convergent step-size interval.
The example in Proposition~\ref{prop:sharp}
shows that \eqref{eq:intro_threshold} is sharp for a global
convergence guarantee from every initialization
with $0<\eta<1/L$: when this condition fails, the example
has a nonoptimal ProjGD fixed point for every step size
in this range.

\textbf{Contribution 3: uniform improvement and global linear convergence.}
We extend the critical-point argument to arbitrary ProjGD iterates.
Under the additional rank-$2r$ restricted smoothness assumption,
Lemma~\ref{lemma:uniform_progress} shows that, for every step size in
\eqref{eq:intro_stepsize}, there is an explicit constant $b>0$,
independent of $\bX$, such that every best rank-$r$ projection satisfies
\[
\|\bX^+-\bX\|_F\geq b\|\bX-\bX_*\|_F,
\qquad \bX^+=\Pr(\bX-\eta\nabla f(\bX)).
\]
This uniform estimate connects the critical-point characterization
to convergence of the actual iterates.
Together with the sufficient-decrease estimate in
Lemma~\ref{lemma:descent2}, it yields the global convergence
guarantee of Theorem~\ref{thm:main1}: there is an explicit
$\gamma\in(0,1)$ such that
\begin{equation}\label{eq:intro_rate}
f(\bX^{(k)})-f(\bX_*)
\leq (1-\gamma)^k\bigl[f(\bX^{(0)})-f(\bX_*)\bigr].
\end{equation}
Neither the step-size interval nor the linear rate depends on
the ratio $\sigma_1(\bX_*)/\sigma_{r_*}(\bX_*)$.
To prove uniform improvement, we distinguish two cases: a large
tangential gradient directly forces a substantial ProjGD update, while
a small tangential gradient implies a sufficiently large normal
component, so the ProjGD update is still substantial.

The critical-point characterization is algorithm-independent and
may therefore inform analyses of other rank-constrained methods
whose accumulation points satisfy the same first-order condition.
Here our main algorithmic application is ProjGD.
As a secondary application, Lemma~\ref{lemma:localminimizer}
constructs a rank-preserving curve with a negative second derivative,
and Theorem~\ref{thm:landscape} gives the no-spurious-local-minima
property for $f$ over $\mathcal M_r$ under \eqref{eq:intro_threshold},
matching the known sharp asymmetric threshold~\cite{mcrae2026matrixlasso}.

\subsection{Relationship with earlier guarantees}

Ha, Liu, and Barber~\cite{doi:10.1137/18M1231675} relate
second-order stationary points of the factored objective to ProjGD
fixed points and derive optimality guarantees, including for
overestimated search ranks. Our critical-point characterization
quantifies the normal gradient at a nonoptimal critical point and
gives an explicit lower step-size threshold for excluding nonoptimal ProjGD fixed points. The uniform improvement estimate in
Lemma~\ref{lemma:uniform_progress} extends this argument
to the ProjGD iterations.

The landscape of low-rank recovery under restricted isometry has also been
studied extensively. Earlier work established benign landscapes for
low-rank matrix sensing and related recovery problems
\cite{NIPS2016_b139e104,pmlr-v54-park17a,pmlr-v70-ge17a}, with
extensions to more general low-rank objectives and nonlinear measurements
\cite{Zhu2018,pmlr-v130-bi21a}. Earlier sharp restricted-isometry
results include \cite{NEURIPS2018_f8da71e5,JMLR:v20:19-020}.
Subsequent work obtained sharp rank-dependent guarantees
for positive-semidefinite problems
\cite{Zhang2021GeneralLM,zhang2022improved} and guarantees for asymmetric, noisy,
and approximately stationary problems~\cite{zhang2025noisy}.
Recently, McRae and Zhang~\cite{mcrae2026matrixlasso} have already
proved the sharp asymmetric rank-overparameterized threshold and
provided matching examples.
Their work also treats nuclear-norm regularization and statistical
recovery, which are outside the scope of our present noiseless analysis.
Our Theorem~\ref{thm:landscape} obtains the corresponding
no-spurious-local-minima property for $f$ over $\mathcal M_r$
at this threshold as a secondary consequence of the critical-point
characterization.

Global convergence guarantees for iterative hard thresholding (IHT),
also called singular-value projection in the low-rank setting, have been
studied extensively. Early results established linear recovery for
least-squares matrix sensing and related quadratic problems
\cite{NIPS2010_08d98638,TannerWei2013}.
More general analyses were later developed under restricted strong
convexity and smoothness. Jain, Tewari, and Kar
\cite{JainTewariKar2014} analyze hard-thresholding methods for general
high-dimensional objectives and extend their framework to low-rank
matrix recovery, while Liu and Barber \cite{LiuBarber2020} give
convergence guarantees for general iterative thresholding schemes,
including hard singular-value thresholding, under restricted curvature
conditions. Thus, global convergence of IHT is not limited to quadratic
sensing losses.

Zhang, Bi, and Lavaei~\cite[Theorem~3]{Zhang2021GeneralLM}
also establish global linear convergence of SVP for general
twice continuously differentiable objectives under
$\delta$-RIP$_{2r,2r}$ with $\delta<1/3$, using
$\eta=(1+\delta)^{-1}$. Their Hessian bounds apply at base points
and in directions of rank at most $2r$, and their target is an
unconstrained critical point of rank at most $r$.
Under A1$'$--A2, Theorem~\ref{thm:main1} gives the rank-dependent
condition $\delta<\sqrt{r/r_*}/(1+\sqrt{r/r_*})$ and the step-size
interval \eqref{eq:intro_stepsize}; when $r=r_*$, this permits
$\delta<1/2$. 

For hard singular-value thresholding, Liu and Barber
\cite[Section~5]{LiuBarber2020} use relative concavity to obtain restricted-optimality guarantees at step size $1/L$
relative to arbitrary rank-$r_*$ comparators when $r>(L/\mu)^2r_*$.
A2 supplies additional information: the comparator is an unconstrained
minimizer with $\nabla f(\bX_*)=0$. Under A1$'$--A2, specializing
their projection inequality already yields global linear convergence
for suitable fixed step sizes when $r>\frac14(L/\mu)^2r_*$;
Appendix~\ref{sec:rc_comparison} gives the short derivation.
The analysis in this paper reaches the less
restrictive condition $r>\frac14(L/\mu-1)^2r_*$ and the step-size
interval \eqref{eq:intro_stepsize}. In particular, this paper combines both
curvature bounds with the matrix geometry, yielding the sharp
critical-point characterization and fixed-point threshold, and then
the uniform improvement estimate of Lemma~\ref{lemma:uniform_progress}.

Additional projection-based guarantees include
\cite{chen2015fast,doi:10.1137/17M1141394}.
For general closed, possibly nonconvex constraint sets,
Olikier and Waldspurger~\cite{olikier2025projected} show that
accumulation points of their backtracking projected-gradient method
are Bouligand stationary and, when the gradient is locally Lipschitz
continuous, proximally stationary. Their result establishes stationarity under much more general
assumptions. Here, restricted smoothness and strong convexity together
with A2 yield the quantitative uniform improvement estimate in
Lemma~\ref{lemma:uniform_progress}, and hence the global linear rate
for ordinary fixed-step ProjGD in Theorem~\ref{thm:main1}.

Other approaches include optimization on fixed- and bounded-rank sets \cite{vandereycken2013lowrank,uschmajew2020critical,
levin2023finding,olikier2023apocalypse}
and factorized methods with preconditioning
\cite{DBLP:journals/jmlr/ZhangFZ23,10.5555/3546258.3546408}.
These works provide complementary theoretical analyses and
algorithmic approaches to low-rank optimization.

\subsection{Organization}
Section~\ref{sec:background} sets up the assumptions and the algorithms.
Section~\ref{sec:certificate} develops the critical-point characterization
and a fixed-point guarantee of the ProjGD algorithm.
Section~\ref{sec:algorithms} proves deterministic global linear
convergence of ProjGD.
Section~\ref{sec:landscape} presents an alternative proof of the known sharp landscape threshold, together with a sharpness example.
Section~\ref{sec:balanced} discusses the relation with Burer–Monteiro
factorization. Section~\ref{sec:discussion} addresses limitations
and future directions. Auxiliary proofs are collected in the appendices.

\section{Problem setting and algorithms}\label{sec:background}
\subsection{Notation}\label{sec:notation}

For a matrix $\bX\in\reals^{n_1\times n_2}$ of rank $q$, write its compact singular value decomposition as $\bX=\bU_{\bX}\bSigma_{\bX}\bV_{\bX}^T$, where $\bU_{\bX}\in\reals^{n_1\times q}, \bV_{\bX}\in\reals^{n_2\times q}, \bSigma_{\bX}\in\reals^{q\times q}$. We order its singular values nonincreasingly and set $\sigma_j(\bX)=0$ for $j>q$. The map $\Pr(\bY)$ denotes any best Frobenius-norm approximation of $\bY$ of rank at most $r$, obtained by truncating its singular value decomposition; if $\rank(\bY)\leq r$, then $\Pr(\bY)=\bY$. 

For a matrix $\bU$ with orthonormal columns, $\bU^\perp$ denotes an orthonormal basis of its orthogonal complement. Write $[\bZ]_{\bU,\bV}=\bU\bU^T\bZ\bV\bV^T$. For a matrix $\bX$ of actual rank $q>0$, $T(\bX)=\{\bX\bA+\bB\bX: \bA\in\reals^{n_2\times n_2},\bB\in\reals^{n_1\times n_1}\}$ denotes the tangent space to the set of rank-$q$ matrices at $\bX$. Its orthogonal projectors are
\begin{align*}
P_{T(\bX)}\bZ
&=\bU_{\bX}\bU_{\bX}^T\bZ+\bZ\bV_{\bX}\bV_{\bX}^T
-\bU_{\bX}\bU_{\bX}^T\bZ\bV_{\bX}\bV_{\bX}^T,\\
P_{T(\bX)^\perp}\bZ
&=(\bI-\bU_{\bX}\bU_{\bX}^T)\bZ(\bI-\bV_{\bX}\bV_{\bX}^T).
\end{align*}
{At $\bX=0$, we set $T(0)=\{0\}$ and $P_{T(0)}=0$.}
These are the standard fixed-rank tangent projectors \cite{Absil2015}. We write $\|\cdot\|$ for the matrix spectral norm, $\|\cdot\|_F$ for the Frobenius norm, and $\|\cdot\|_*$ for the nuclear norm.

For $f:\reals^{n\times n}\to\reals$, we write $\nabla f(\bX)\in\reals^{n\times n}$ for its Euclidean gradient, $P_{T(\bX)}\nabla f(\bX)$ for its Riemannian gradient, and $\nabla^2 f(\bX)[\bE]$ for the directional derivative of $\nabla f$ at $\bX$ in the direction $\bE$. Thus $\langle\nabla^2 f(\bX)[\bE],\bE\rangle$ is the second directional derivative of $t\mapsto f(\bX+t\bE)$ at $t=0$. We refer to $P_{T(\bX)^\perp}\nabla f(\bX)$ as the \emph{normal component of the Euclidean gradient}, or simply the \emph{normal component of the gradient} when the context is clear.

For any $\bX$ with $\rank(\bX)\leq r$, let $\mathcal U_{\bX}=\operatorname{span}(\operatorname{col}(\bX),\operatorname{col}(\bX_*))$ and $\mathcal V_{\bX}=\operatorname{span}(\operatorname{row}(\bX),\operatorname{row}(\bX_*))$. Choose orthonormal basis matrices $\bU_+\in\reals^{n\times d_U}$ and $\bV_+\in\reals^{n\times d_V}$ for these subspaces, where $d_U=\dim\mathcal U_{\bX}$ and $d_V=\dim\mathcal V_{\bX}$, and define
\[
\mathcal S_{\bX}=\{\bU_+\bA\bV_+^T:\bA\in\reals^{d_U\times d_V}\},
\qquad
P_{\mathcal S_{\bX}}(\bZ)=\bU_+\bU_+^T\bZ\bV_+\bV_+^T.
\]
The definition of $P_{\mathcal S_{\bX}}$ is independent of the chosen orthonormal bases. Every element of $\mathcal S_{\bX}$ has rank at most $r+r_*$ because both joint subspaces have dimension at most $r+r_*$. Moreover $P_{\mathcal S_{\bX}}$ commutes with $P_{T(\bX)}$: the column and row spaces of $\bX$ are contained in $\mathcal U_{\bX}$ and $\mathcal V_{\bX}$.

\subsection{Assumptions}
The critical-point characterization uses a rank-$(r+r_*)$ restricted second-derivative condition. For nonlinear objectives this is the natural second-order analogue of RIP; for matrix sensing it is standard low-rank RIP after normalization.

\textbf{Assumption A1 (Rank-$(r+r_*)$ restricted smoothness and strong convexity / RIP-type condition).}
The function $f:\reals^{n\times n}\to\reals$ is twice continuously differentiable. For every $\bX$ and every $\bE$ with $\rank(\bE)\leq r+r_*$, constants $0<\mu\leq L$ satisfy
\begin{equation}\label{eq:assumption1}
\mu\|\bE\|_F^2\leq
\langle\nabla^2f(\bX)[\bE],\bE\rangle
\leq L\|\bE\|_F^2.
\end{equation}
This implies that $f$ has rank-$(r+r_*)$ restricted $L$-smoothness and $\mu$-strong convexity. Define the normalized parameter for restricted smoothness and strong convexity
$\delta:=\frac{L-\mu}{L+\mu}\in[0,1).
$ After scaling the objective by $2/(L+\mu)$, \eqref{eq:assumption1} becomes the normalized bound $(1-\delta)\|\bE\|_F^2\leq\langle\nabla^2f(\bX)[\bE],\bE\rangle\leq(1+\delta)\|\bE\|_F^2$. If $f(\bX)=\frac12\|\mathcal A(\bX-\bX_*)\|_2^2$, this is rank-$(r+r_*)$ RIP for $\mathcal A$.

\textbf{Assumption A1$'$ (Additional rank-$2r$ restricted smoothness for ProjGD).}
Assumption A1 holds with constants $0<\mu\leq L$. In addition, for every $\bX$ and every $\bE$ with $\rank(\bE)\leq2r$, we assume
\[
\langle\nabla^2 f(\bX)[\bE],\bE\rangle\leq L\|\bE\|_F^2.
\]
Thus A1$'$ imposes no rank-$2r$ restricted strong-convexity condition: the lower bound with constant $\mu$ is required only through rank $r+r_*$. The additional smoothness range is used to control the difference between consecutive rank-$r$ iterates in the sufficient-decrease argument.

Throughout, \(L\) and \(\mu\) denote the restricted smoothness and strong convexity constants in A1. Whenever A1'  is invoked, the same \(L\) also bounds rank-\(2r\) directions.

\textbf{Assumption A2 (Low-rank unconstrained minimizer).}
There exists an unconstrained minimizer $\bX_*\in\argmin_{\bX\in\reals^{n\times n}}f(\bX)$ of rank $r_*$, with $1\leq r_*\leq r<n$. Thus $\nabla f(\bX_*)=0$. Under A1, $\bX_*$ is the unique minimizer over $\mathcal M_r$: the difference between any rank-$r$ point and $\bX_*$ has rank at most $r+r_*$, so restricted strong convexity applies.

\textbf{Assumption A3 (Nonempty ProjGD step-size interval).}
Set $\rho=\sqrt{r/r_*}$. We assume
\begin{equation}\label{eq:A3delta}
\delta<\frac{\rho}{1+\rho}.
\end{equation}
Equivalently, this is 
\[
\frac L\mu<1+2\sqrt{\frac r{r_*}}
,\,\,\text{or}\,\,
r>\frac14\left(\frac L\mu-1\right)^2r_*.
\]

\textbf{Discussion of the assumptions.}
A1 is substantially weaker than global strong convexity/smoothness: high-rank directions are unrestricted. In underdetermined matrix sensing, $\mathcal A$ may therefore have a large ambient nullspace while still satisfying A1. The critical-point characterization and its landscape consequence use only rank $r+r_*$, consistent with the order appearing in \cite{zhang2022improved,zhang2025noisy,mcrae2026matrixlasso}. The ProjGD sufficient-decrease argument additionally requires rank-$2r$ restricted smoothness under A1$'$, because consecutive rank-$r$ iterates may differ by a matrix of rank as large as $2r$. Rank-$2r$ restricted strong convexity is not needed. By contrast, the critical-point characterization and landscape analysis use only rank-$(r+r_*)$ restricted smoothness and strong convexity.

\subsection{Stationarity and rank-deficient boundary points}
Let the constraint set of \eqref{eq:problem} be
\[
\mathcal{M}_r=\{\bX\in\reals^{n\times n}:\rank(\bX)\leq r\}.
\]
For \(\bX\in\mathcal M_r\), let \(T(\bX)\) be the tangent space defined in Section~\ref{sec:notation}. For a full-rank ($\rank(\bX)=r$) first-order stationary point of the constrained problem \eqref{eq:problem}, referred to below as a \emph{rank-$r$ critical point}, the first-order necessary condition is
\[
P_{T(\bX)}\nabla f(\bX)=0.
\]
We call such a critical point a \emph{strict saddle} if there exists a smooth curve \(\bX(t)\in \mathcal{M}_r\) through \(\bX\) for which
\[
\bX(0)=\bX,\,\,\left.\frac{\mathrm d^2}{\mathrm dt^2}f(\bX(t))\right|_{t=0}<0.
\]
For a point of rank $q<r$, the fixed-rank tangent-space condition
is not the first-order optimality condition for $\mathcal M_r$: a rank-one
perturbation preserves the rank constraint, so a local minimizer at such a point must
satisfy $\nabla f(\bX)=0$. Under A1--A2 the latter is possible
only at $\bX_*$. At rank-deficient points, we show that any nonoptimal point in $\mathcal M_r$ is neither a ProjGD fixed point nor a local minimizer.

\subsection{Algorithms}
\textbf{Projected gradient descent (ProjGD)  algorithm} The projected gradient descent algorithm  treats \eqref{eq:problem} as a constrained optimization problem. It iteratively performs a gradient step and then projects the result to satisfy the constraints:
\begin{equation}\label{eq:algorithm}
\bX^{(t+1)}=\Pr\big(\bX^{(t)}-\eta\nabla f(\bX^{(t)})\big),
\end{equation}
where $\Pr$ is the best rank-$r$ approximation defined in Section~\ref{sec:notation}.

\textbf{Factored gradient descent (FGD) algorithm} 
{FGD uses the factorization $\bX=\bU\bV^T$, with
$\bU,\bV\in\reals^{n\times r}$, introduced in the low-rank approach
of Burer and Monteiro~\cite{Burer:2003ve}, and applies gradient descent
to the factor variables
\cite{zheng2016convergence,Boumal2018,zhang2022improved,pmlr-v54-park17a,10.5555/3157382.3157407}.}
Denoting the gradient with respect to
 $\bU, \bV, \bX$  by $\nabla_{\bU}, \nabla_{\bV}, \nabla_{\bX}$, the FGD iteration is:
\begin{align}
\bU^{(t+1)} &= \bU^{(t)}-\eta\nabla_{\bU}f|_{\bX=\bX^{(t)}}=
\bU^{(t)} - \eta\nabla_{\bX}f(\bX^{(t)}) \bV^{(t)} \label{eq:algorithm_factor1}, \\\nonumber
\bV^{(t+1)} &= \bV^{(t)}-\eta\nabla_{\bV}f|_{\bX=\bX^{(t)}}= \bV^{(t)} - \eta\big[\nabla_{\bX}f(\bX^{(t)})\big]^T \bU^{(t)}, \\\nonumber
\bX^{(t+1)} &= \bU^{(t+1)} (\bV^{(t+1)})^T.
\end{align}
In Section~\ref{sec:balanced} we relate second-order stationary
points of this factorization to ProjGD fixed points.

\section{Critical-point characterization}\label{sec:certificate}
We establish a quantitative critical-point characterization for nonoptimal
rank-$r$ critical points in terms of the normal component of the gradient and the
$r$-th singular value. We then explain how this critical-point characterization
excludes nonoptimal fixed points of ProjGD. Section~\ref{sec:algorithms} upgrades this static obstruction to a uniform one-step progress bound and deterministic linear convergence.

\subsection{The key inequality}
The goal of this subsection is to convert the rank-$(r+r_*)$ restricted smoothness and strong convexity assumptions into a quantitative lower bound on the normal component of the Euclidean gradient. A difficulty is that A1 controls the second derivative only along low-rank directions. We therefore restrict the second-derivative operator to the joint row and column spaces generated by $\bX$ and $\bX_*$. Every matrix in the resulting linear subspace $\mathcal S_{\bX}$ has rank at most $r+r_*$, so A1 controls this compressed operator on the entire subspace. Lemma~\ref{lemma:gradient} gives the corresponding centered-gradient estimate. 

\begin{lemma}[Compressed centered-gradient estimate]\label{lemma:gradient}
Assume A1--A2 and let $\rank(\bX)\leq r$. Put $\bE=\bX-\bX_*$ and let $\mathcal S_{\bX}$ be the joint row/column subspace defined in Section~\ref{sec:notation}. Then
\begin{equation}\label{eq:compressed_gradient}
\left\|P_{\mathcal S_{\bX}}\nabla f(\bX)-\frac{L+\mu}{2}\bE\right\|_F
\leq\frac{L-\mu}{2}\|\bE\|_F.
\end{equation}
Equivalently, with $c=(L+\mu)/2$ and $\kappa_0=(L-\mu)/(L+\mu)$,
\[
\left\|-\frac1cP_{\mathcal S_{\bX}}\nabla f(\bX)+\bX-\bX_*\right\|_F
\leq\kappa_0\|\bX-\bX_*\|_F.
\]
\end{lemma}
This shows that the gradient at a spurious point is compressed
onto the joint row and column spaces of $\bX$ and $\bX_*$, centered at
$\bX-\bX_*$, and the
remainder is bounded through restricted isometry. Although an equivalent estimate for $r=r_*$ appears in the proof of \cite[Theorem~4]{Zhang2021GeneralLM}, we provide a self-contained proof here for completeness. The compression $P_{\mathcal S_{\bX}}$ in Lemma~\ref{lemma:gradient} is essential due to the restricted rank assumption. Although A1 need not control the second derivative in arbitrary high-rank directions, its restriction to $\mathcal S_{\bX}$ has spectrum in $[\mu,L]$ because every direction in that linear subspace has rank at most $r+r_*$. Thus Lemma~\ref{lemma:gradient} extends the usual centered-gradient estimate from global to restricted smoothness and strong convexity.

The next lemma is an algebraic bound on a normal perturbation $\bZ$.
Its hypothesis has the same form as the centered-gradient estimate,
so $\bZ$ plays the role of a scaled compressed gradient at a rank-$r$ critical point of $f$. The factor $\sqrt{r/r_*}$ strengthens the critical-point characterization under overparameterization. The proof uses completion of the square and
a singular-value counting argument.

\begin{lemma}[A matrix inequality]\label{lemma:orthogonalrate0}
Let $\rank(\bX)=r$, $\rank(\bX_*)\leq r_*$, $1\leq r_*\leq r$, and $\bX\neq\bX_*$. Suppose $0<\kappa_0<1$ and
\begin{equation}\label{eq:bZ1}
P_{T(\bX)}\bZ=0,\qquad
\|\bZ+\bX-\bX_*\|_F\leq\kappa_0\|\bX-\bX_*\|_F.
\end{equation}
Then
\[
\|\bZ\|\geq\sqrt{\frac r{r_*}}\frac{1-\kappa_0^2}{\kappa_0}\sigma_r(\bX).
\]
\end{lemma}

{\begin{proof}
Let $a=1-\kappa_0^2>0$. Completing the square in \eqref{eq:bZ1} gives
\begin{equation}\label{eq:completed_square}
a\left\|\bX+\frac{\bZ}{a}-\bX_*\right\|_F^2
\leq\frac{\kappa_0^2}{a}\|\bZ\|_F^2.
\end{equation}
The normality condition $P_{T(\bX)}\bZ=0$ puts $\bX$ and $\bZ$ in orthogonal row and column blocks. Thus the singular values of $\bX+\bZ/a$ are the combined singular values of $\bX$ and $\bZ/a$. By Eckart--Young \cite{eckart1936approximation}, its best rank-$r_*$ error is no larger than the error to $\bX_*$. Suppose the best rank-$r_*$ approximation of $\bX+\bZ/a$ retains $j$ singular values from the $\bZ/a$ block and $r_*-j$ from the $\bX$ block. Let $S_J$ be the sum of squares of the retained singular values of $\bZ$, and $S_{J^c}$ the sum of squares of the remaining ones. Since at least $r-r_*+j$ singular values of $\bX$ are discarded, \eqref{eq:completed_square} implies
\[
\begin{gathered}
a^2(r-r_*+j)\sigma_r^2(\bX)+S_{J^c}
\leq\kappa_0^2(S_J+S_{J^c}),\\
\quad\text{hence}\quad
a^2(r-r_*+j)\sigma_r^2(\bX)+aS_{J^c}\leq\kappa_0^2S_J.
\end{gathered}
\]
If $j=0$, then $J=\emptyset$, so $S_J=0$ and the inequality implies that $S_{J^c}=0$. On the other hand, $J=\emptyset$ means that $S_{J^c}$ is the sum of all squared singular values of $\bZ$. As a result, $\bZ=0$.
Then the inequality in  \eqref{eq:bZ1} forces $\bX_*=\bX$, a contradiction. Thus $1\leq j\leq r_*$, and $S_J\leq j\|\bZ\|^2$ yields
\[
\|\bZ\|^2\geq\frac{a^2}{\kappa_0^2}\frac{r-r_*+j}{j}\sigma_r^2(\bX)
\geq\frac{a^2}{\kappa_0^2}\frac r{r_*}\sigma_r^2(\bX).
\]
Taking square roots proves the claim. %
\end{proof}}

Lemma~\ref{lemma:orthogonalrate0} says that if a perturbation $\bZ$ makes a rank-$r$ matrix $\bX$ substantially closer to a rank-$r_*$ target $\bX_*$ and is orthogonal to $T(\bX)$, then $\bZ$ must contain a singular value large compared with $\sigma_r(\bX)$. The factor $\sqrt{r/r_*}$ is the source of the improvement obtained from rank overparameterization. 

The two preceding lemmas now combine to prove properties of a rank-$r$ critical point. Lemma~\ref{lemma:gradient} says that the  gradient $\nabla f(\bX)$ (when restricted to $\mathcal S_{\bX}$) points approximately in the error direction $\bX-\bX_*$, while Lemma~\ref{lemma:orthogonalrate0} says that any $\bZ$ with this approximation property cannot be spectrally small. Their combination gives Theorem~\ref{thm:gap}, the quantitative core of the proof of this Section. It depends on restricted smoothness and strong convexity parameters $L,\mu$ and the rank ratio $r/r_*$; we first use it to rule out spurious ProjGD fixed points, then to control finite steps in Section~\ref{sec:algorithms}.

\begingroup

\begin{theorem}[Critical-point characterization]\label{thm:gap}
Under Assumptions A1--A2, suppose $0<\mu<L$ and let $\bX\in\mathcal{M}_r$ be a nonoptimal rank-$r$ critical point, so $P_{T(\bX)}\nabla f(\bX)=0$ and $\bX\neq\bX_*$. Then 
\begin{equation}\label{eq:gap}
\|P_{T(\bX)^\perp}\nabla f(\bX)\|
\geq\sqrt{\frac{r}{r_*}}\,\frac{2L\mu}{L-\mu}\,\sigma_r(\bX).
\end{equation}
\end{theorem}
\endgroup

\begin{proof}
Let $\bG=\nabla f(\bX)$, $\widetilde{\bG}=P_{\mathcal S_{\bX}}\bG$, $c=(L+\mu)/2$, and $\kappa_0=(L-\mu)/(L+\mu)$. Since $P_{T(\bX)}\bG=0$ and $P_{\mathcal S_{\bX}}$ commutes with $P_{T(\bX)}$, we also have $P_{T(\bX)}\widetilde{\bG}=0$. Set $\bZ=-\widetilde{\bG}/c$. Lemma~\ref{lemma:gradient} gives
\[
P_{T(\bX)}\bZ=0,
\qquad
\|\bZ+\bX-\bX_*\|_F\leq\kappa_0\|\bX-\bX_*\|_F.
\]
Applying Lemma~\ref{lemma:orthogonalrate0} and using
\[
\frac{1-\kappa_0^2}{\kappa_0}
=\frac{4L\mu}{(L+\mu)(L-\mu)}
\]
yields \begin{equation}\label{eq:gap_compressed}
\|P_{\mathcal S_{\bX}}\nabla f(\bX)\|
\geq\sqrt{\frac{r}{r_*}}\,\frac{2L\mu}{L-\mu}\,\sigma_r(\bX).
\end{equation}
Because orthogonal compression $P_{\mathcal S_{\bX}}$ cannot increase spectral norm and $P_{T(\bX)}\bG=0$,
\[
\|P_{T(\bX)^\perp}\bG\|=\|\bG\|\geq\|\widetilde{\bG}\|,
\]
which gives \eqref{eq:gap}. 
\end{proof}

{Under Assumption A3, the coefficient
$\sqrt{r/r_*}\,2L\mu/(L-\mu)$ in \eqref{eq:gap} exceeds $L$.}
As shown in the next subsection, this bound
excludes nonoptimal ProjGD fixed points for sufficiently
large step sizes and ensures a nonempty step-size interval
for the global convergence result in Section~\ref{sec:algorithms}.
The same bound is also used in Section~\ref{sec:landscape} to rule out
nonoptimal local minima.

\subsection{Fixed-point interpretation and the sharp step-size threshold}
{For $\eta>0$, we call $\bX$ a \emph{ProjGD fixed point} if}
\begin{equation}\label{eq:fixedpoint_def}
\bX\in\argmin_{\rank(\bY)\leq r}
\|\bY-[\bX-\eta\nabla f(\bX)]\|_F.
\end{equation}
This definition allows ties in the truncated SVD. 
The fixed-point criterion below is due to Ha, Liu, and
Barber~\cite{doi:10.1137/18M1231675}. Under their restricted
curvature assumptions and Assumption A2, their Theorem~2 guarantees uniqueness of the fixed point
when $L<2\mu$ and $1/(2\mu)<\eta\leq1/L$. In comparison, our Theorem~\ref{thm:gap} provides a stronger,
rank-dependent lower bound on the normal component of
the gradient at nonoptimal critical points.
Combining this bound with the existing fixed-point
characterization, we obtain the less restrictive
sufficient condition~\eqref{eq:intro_threshold}.
In particular, this allows $L<3\mu$ when $r=r_*$,
with further improvements as the search rank $r$
increases relative to $r_*$.

\begingroup

\begin{theorem}[Fixed-point criterion and exclusion]\label{thm:fixedpoint}
Assume A1--A2 and fix $\eta>0$.
\begin{enumerate}
\item If $\rank(\bX)=r$, then $\bX$ is a ProjGD fixed point
if and only if
\begin{equation}\label{eq:fixedpoint_criterion}
P_{T(\bX)}\nabla f(\bX)=0,
\qquad
\eta\|P_{T(\bX)^\perp}\nabla f(\bX)\|\leq\sigma_r(\bX).
\end{equation}
If $\rank(\bX)<r$, then $\bX$ is a ProjGD fixed point if and
only if $\nabla f(\bX)=0$.
\item If $L>\mu$, then, whenever
\begin{equation}\label{eq:fixedpoint_eta}
\eta>\eta_{\mathrm{crit}}:=
\sqrt{\frac{r_*}{r}}\frac{L-\mu}{2L\mu},
\end{equation}
$\bX_*$ is the \emph{only} ProjGD fixed point over
$\mathcal M_r$. If $L=\mu$, the same conclusion holds for every
$\eta>0$.
\end{enumerate}
\end{theorem}
\endgroup
\begin{proof}
Let $\bG=\nabla f(\bX)$.
If $\rank(\bX)=r$, the truncated-SVD characterization of best
rank-$r$ approximations shows that $\bX$ is a best approximation of
$\bX-\eta\bG$ when $-\eta\bG$ lies in
$T(\bX)^\perp$ and its largest singular value does not exceed
$\sigma_r(\bX)$. This is \eqref{eq:fixedpoint_criterion}.
If $\rank(\bX)<r$ and $\bX$ is a best rank-$r$ approximation of
$\bX-\eta\bG$, then that latter matrix must equal $\bX$:
a best approximation of rank less than $r$ leaves no nonzero
singular value unretained. Thus $\bG=0$; the converse is immediate.

For any $\bX\in\calM_r$ with $\bX\ne\bX_*$, let $\bE=\bX-\bX_*$. Restricted
strong convexity along $\bE$, whose rank is at most $r+r_*$, gives
\begin{equation}\label{eq:fixedpoint_nonzero_gradient}
\langle\nabla f(\bX),\bE\rangle
=\langle\nabla f(\bX)-\nabla f(\bX_*),\bE\rangle
\geq\mu\|\bE\|_F^2>0.
\end{equation}
Therefore a nonoptimal rank-deficient point is never fixed.
At rank $r$ with $L>\mu$, if the first condition of
\eqref{eq:fixedpoint_criterion} holds, Theorem~\ref{thm:gap}
implies
$\eta\|\bG\|\geq\eta(2L\mu/(L-\mu))\sqrt{r/r_*}\,
\sigma_r(\bX)>\sigma_r(\bX)$ whenever
$\eta>\eta_{\mathrm{crit}}$. This contradicts the second condition.

If $L=\mu$, Lemma~\ref{lemma:gradient} gives
$P_{\mathcal S_{\bX}}\nabla f(\bX)=L(\bX-\bX_*)$.
A hypothetical nonoptimal rank-$r$ critical point would then satisfy
$P_{T(\bX)}(\bX-\bX_*)=0$. In an SVD coordinate system for $\bX$,
this equality says $\bX_*=\bX+\bW$ with
$\bW\in T(\bX)^\perp$. The two terms occupy orthogonal row and
column blocks, so
$\rank(\bX_*)=r+\rank(\bW)\geq r$.
Since $\rank(\bX_*)=r_*\leq r$, necessarily
$\bW=0$ and $\bX_*=\bX$, a contradiction.
Finally, $\nabla f(\bX_*)=0$, so $\bX_*$ is always fixed.
\end{proof}

For another comparison, the relative-concavity inequality of Liu and
Barber~\cite{LiuBarber2020}, combined with A1--A2, gives at any
ProjGD fixed point, with $\bE=\bX-\bX_*$,
\[
\eta\mu\|\bE\|_F^2
\leq\eta\langle\nabla f(\bX),\bE\rangle
\leq\tfrac12\sqrt{\frac{r_*}{r}}\,\|\bE\|_F^2.
\]
Thus this argument already excludes nonoptimal fixed points when
$\eta>\sqrt{r_*/r}/(2\mu)$. Theorem~\ref{thm:fixedpoint} lowers
this threshold by the factor $1-\mu/L$ by using both curvature
bounds through Lemma~\ref{lemma:gradient};
Proposition~\ref{prop:sharp} attains the resulting threshold.
This distinguishes the consequence of the zero-gradient target
assumption from the further gain obtained using restricted upper
curvature.

Theorem~\ref{thm:fixedpoint} is a statement about fixed points of ProjGD:
uniqueness of fixed points alone does not imply convergence, let alone
a linear convergence rate.
The next section analyzes the convergence of the ProjGD algorithm.

\section{Deterministic global linear convergence of ProjGD}\label{sec:algorithms}
Section~\ref{sec:certificate} gives a quantitative characterization
of exact critical points. Combined with the fixed-point criterion,
it excludes nonoptimal ProjGD fixed points for
$\eta>\eta_{\mathrm{crit}}$.
The main step here is to extend this  argument to a uniform
finite-step progress estimate valid at every iterate:
\[
\|\bX^+-\bX\|_F\geq b\|\bX-\bX_*\|_F,
\qquad \bX^+=\Pr(\bX-\eta\nabla f(\bX)),
\]
with $b>0$ independent of $\bX$.
Lemma~\ref{lemma:uniform_progress} supplies this estimate, and
Theorem~\ref{thm:main1} combines it with the sufficient-decrease estimate in Lemma~\ref{lemma:descent2} 
to prove deterministic global linear convergence.

\subsection{Global convergence of projected gradient descent}\label{sec:global}

The additional rank-$2r$ restricted smoothness in
Assumption~A1$'$ is needed because the difference between
two consecutive ProjGD iterates can have rank as large
as $2r$. The next lemma gives sufficient decrease and the projection estimate used in the uniform improvement argument.

\begin{lemma}[Sufficient decrease and projection estimates]\label{lemma:descent2}
Assume A1$'$. Let $\bX\in\calM_r$ and $0<\eta<1/L$, and let
$\bX^+=\Pr(\bX-\eta\nabla f(\bX))$ denote one ProjGD update.
Then 
\begin{equation}\label{eq:descent2}
f(\bX)-f(\bX^+)\geq \frac{1}{2}\Big(\frac{1}{\eta}- L\Big)\|\bX-\bX^+\|_F^2.
\end{equation}
The term $\|\bX-\bX^+\|_F$ in the RHS of \eqref{eq:descent2} is bounded below by
\begin{equation}\label{eq:descent22}
\|\bX-\bX^+\|_F\geq 
\max\Big(\frac{\eta\|\nabla f(\bX)\|-\sigma_r(\bX)}{2}, \frac{\eta\|P_{T(\bX)}\nabla f(\bX)\|_F}{\sqrt{10}}\Big)
\end{equation}
\end{lemma}

\begin{lemma}[Uniform improvement of ProjGD]\label{lemma:uniform_progress}
Assume A1$'$--A3, and let the step size $\eta$ satisfy
\begin{equation}\label{eq:rangeassumption}
\sqrt{\frac{r_*}{r}}\frac{L-\mu}{2L\mu}<\eta<\frac1L.
\end{equation}
Set $\kappa_0=(L-\mu)/(L+\mu)$ and $\eta_0=\eta(L+\mu)/2$. Choose $\epsilon>0$ so that $\hat\kappa_0=\kappa_0+\epsilon<1$ and $\sqrt{r_*/r}\,\hat\kappa_0/(1-\hat\kappa_0^2)<\eta_0$. Such an $\epsilon$ exists by \eqref{eq:rangeassumption}.
Define
\[
b=\min\left\{\frac{\epsilon\eta_0}{\sqrt{10}},\,
\frac{1-\hat\kappa_0}{2\sqrt{r_*}}
\left(\eta_0-\sqrt{\frac{r_*}{r}}\frac{\hat\kappa_0}{1-\hat\kappa_0^2}\right)\right\}>0.
\]
Then
\begin{equation}\label{eq:uniform_step}
\|\bX^+-\bX\|_F\geq b\|\bX-\bX_*\|_F,
\qquad \bX^+=\Pr(\bX-\eta\nabla f(\bX)),
\end{equation}
for every $\bX\in\mathcal M_r$ and every choice of a best
rank-$r$ approximation.
\end{lemma}

\begin{proof}
At an exact nonoptimal rank-$r$ critical point, the
critical-point characterization gives
$\eta\|P_{T(\bX)^\perp}\nabla f(\bX)\|>\sigma_r(\bX)$,
which excludes a fixed point but does not quantify the update
at a nonstationary iterate.
The estimate \eqref{eq:uniform_step} is immediate at
$\bX=\bX_*$. Otherwise, write
$\bE=\bX-\bX_*$, $\bG=\nabla f(\bX)$, and $c=(L+\mu)/2$.

If $\|P_{T(\bX)}\bG\|_F/c\geq\epsilon\|\bE\|_F$, the tangent estimate
in Lemma~\ref{lemma:descent2} gives
$\|\bX^+-\bX\|_F\geq(\epsilon\eta_0/\sqrt{10})\|\bE\|_F$.
Otherwise, set
$\widetilde{\bZ}=-P_{\mathcal S_{\bX}}\bG/c$ and
$\bN=P_{T(\bX)^\perp}\widetilde{\bZ}$. Lemma~\ref{lemma:gradient}
and the small tangential component imply
\[
\|\bN+\bE\|_F\leq(\kappa_0+\epsilon)\|\bE\|_F
=\hat\kappa_0\|\bE\|_F.
\]
At rank $r$, the matrix inequality underlying the critical-point characterization applies with $\kappa_0$ replaced by $\hat\kappa_0$:
Lemma~\ref{lemma:orthogonalrate0} bounds $\sigma_r(\bX)$ in terms of
$\|\bN\|$.
{To bound $\|\bN\|$ relative to the error without assuming that
$\bN$ has low rank, use $\|\bN+\bE\|_F\leq\hat\kappa_0\|\bE\|_F$,
$P_{T(\bX)}\bN=0$, and $\bX\in T(\bX)$.
Cauchy--Schwarz and spectral--nuclear norm duality give
\begin{align*}
(1-\hat\kappa_0)\|\bE\|_F^2
&\leq-\langle\bN,\bE\rangle
=\langle\bN,P_{T(\bX)^\perp}\bX_*\rangle\\
&\leq\|\bN\|\,\|P_{T(\bX)^\perp}\bX_*\|_*
\leq\sqrt{r_*}\,\|\bN\|\,\|\bE\|_F.
\end{align*}
Here $\rank(P_{T(\bX)^\perp}\bX_*)\leq r_*$ and
$P_{T(\bX)^\perp}\bX_*=-P_{T(\bX)^\perp}\bE$.
Thus $\|\bN\|\geq(1-\hat\kappa_0)\|\bE\|_F/\sqrt{r_*}$.
Together the bounds yield}
\[
\eta\|\bG\|-\sigma_r(\bX)
\geq\frac{1-\hat\kappa_0}{\sqrt{r_*}}
\left(\eta_0-\sqrt{\frac{r_*}{r}}
\frac{\hat\kappa_0}{1-\hat\kappa_0^2}\right)\|\bE\|_F.
\]
The same bound holds when $\rank(\bX)<r$, since then
$\sigma_r(\bX)=0$. The spectral estimate in
Lemma~\ref{lemma:descent2} gives the second branch in the definition
of $b$, proving \eqref{eq:uniform_step} in both cases.
\end{proof}

\begin{theorem}[Deterministic global linear convergence of ProjGD]\label{thm:main1}
Under A1$'$--A3, let $\eta$ satisfy \eqref{eq:rangeassumption},
and choose $\epsilon$ and $b$ as in Lemma~\ref{lemma:uniform_progress}.
Consider the ProjGD algorithm \eqref{eq:algorithm}. For every
initialization $\bX^{(0)}$ with $\rank(\bX^{(0)})\leq r$, the iterates
converge $R$-linearly to $\bX_*$, with
\begin{equation}\label{eq:thm1_rate}
 f(\bX^{(k+1)})-f(\bX_*)\leq
 \left[1-\left(\frac1{\eta L}-1\right)b^2\right]
 \bigl(f(\bX^{(k)})-f(\bX_*)\bigr).
\end{equation}
\end{theorem}

\begin{proof}
Write $\ell_k=f(\bX^{(k)})-f(\bX_*)$. Assumptions A1--A2 give
\[
\frac\mu2\|\bX^{(k)}-\bX_*\|_F^2
\leq\ell_k
\leq\frac L2\|\bX^{(k)}-\bX_*\|_F^2.
\]
Combining the sufficient-decrease estimate in
Lemma~\ref{lemma:descent2} with Lemma~\ref{lemma:uniform_progress} yields
\[
\ell_k-\ell_{k+1}
\geq\frac12\left(\frac1\eta-L\right)b^2
\|\bX^{(k)}-\bX_*\|_F^2
\geq\left(\frac1{\eta L}-1\right)b^2\ell_k,
\]
which proves \eqref{eq:thm1_rate}.
The contraction factor lies in $(0,1)$: since
$0<\epsilon<1$ and $\eta_0\leq\eta L$, the definition of $b$ gives
\[
0<\left(\frac1{\eta L}-1\right)b^2
\leq\frac{\eta L(1-\eta L)}{10}<1.
\]
Finally, with $\gamma=(1/(\eta L)-1)b^2\in(0,1)$, the bounds on
$\ell_k$ and \eqref{eq:thm1_rate} give
\[
\|\bX^{(k)}-\bX_*\|_F
\leq \sqrt{\frac{2\ell_0}{\mu}}\,(1-\gamma)^{k/2}.
\]
Thus the matrices converge $R$-linearly to $\bX_*$.
\end{proof}

\paragraph{{Role of the step size}}
The fixed-point and sufficient-decrease conditions on $\eta$
combine to give the convergence guarantee as follows:
\begin{itemize}
    \item \textbf{Exclusion of nonoptimal fixed points.}
    For $\eta>\eta_{\mathrm{crit}}$, where
    \[
    \eta_{\mathrm{crit}}
    =\sqrt{\frac{r_*}{r}}\frac{L-\mu}{2L\mu},
    \]
    Theorem~\ref{thm:fixedpoint} guarantees that
    $\bX_*$ is the only ProjGD fixed point.
    Thus, every nonoptimal point in $\mathcal M_r$ must change
    under a ProjGD update, although objective decrease
    is not guaranteed by this condition alone.

    \item \textbf{Sufficient decrease.}
    For $0<\eta<1/L$, {Lemma~\ref{lemma:descent2}} ensures
    \[
    f(\bX)-f(\bX^+)
    \geq \frac12\left(\frac1\eta-L\right)
    \|\bX^+-\bX\|_F^2.
    \]
    Hence, every nonzero ProjGD update strictly
    decreases the objective. 
    
    \item \textbf{Global linear convergence.}
    When $\eta_{\mathrm{crit}}<\eta<1/L$,
    both conditions hold. Together with the uniform
progress bound in Lemma~\ref{lemma:uniform_progress},
    they yield deterministic global linear convergence
    from every initialization of rank at most $r$.
\end{itemize}

\textbf{Dependence of the convergence rate.}
The step-size interval in Theorem~\ref{thm:main1} depends on the
restricted curvature constants and the rank ratio, while the
constant $b$ in Lemma~\ref{lemma:uniform_progress} additionally
depends on the intrinsic rank $r_*$, through the factor
$1/\sqrt{r_*}$ in its second branch. The resulting rate is independent of
$\sigma_1(\bX_*)/\sigma_{r_*}(\bX_*)$. The purpose of the
uniform improvement bound is to establish a global linear rate
throughout \eqref{eq:rangeassumption}; its contraction factor
need not improve on other IHT/SVP estimates in overlapping
parameter regimes.

\begin{remark}[Lower step-size bound is sharp]\label{rem:projgd_obstruction}
For the family in Proposition~\ref{prop:sharp}, $\bar\bX-\eta\nabla f(\bar\bX)=\diag(\eta g\bI_{r_*},s\bI_r,0)$ with $g=2\kappa/(\kappa+1)$. Hence $\bar\bX$ remains a best rank-$r$ approximation whenever
\[
0<\eta\leq\frac{s}{g}=\sqrt{\frac{r_*}{r}}\frac{\kappa-1}{2\kappa}=\sqrt{\frac{r_*}{r}}\frac{L-\mu}{2L\mu}.
\]
For $\eta<\eta_{\mathrm{crit}}$, this projection is unique.
At equality, a choice among the tied singular values can still
return $\bar\bX$.
Thus the strict lower step-size condition in
Theorem~\ref{thm:main1} is sharp and cannot be relaxed.
\end{remark}

\paragraph{Local behavior without condition A3}
The rank-dependent condition A3 is a global requirement: it rules out stationary traps far from the solution and enables a uniform global progress estimate. Here we state and prove a complementary local theorem showing that, under A1$'$--A2 alone, every sufficiently accurate initialization (initialization neighborhood depends on $\sigma_{r_*}(\bX_*)$) converges linearly to $\bX_*$ for any $0<\eta<1/L$, including the rank-overestimated case $r_*<r$. 
Its convergence rate depends only on $L/\mu$ and $\eta L$,
but not on the effective condition number
$\kappa_*=\sigma_1(\bX_*)/\sigma_{r_*}(\bX_*)$.
This contrasts with standard, unpreconditioned factored
gradient descent (FGD), whose convergence guarantees
typically deteriorate as $\kappa_*$ increases
\cite{10.5555/3546258.3546408}.
Moreover, rank overparameterization can further slow
down ordinary FGD, potentially leading to sublinear
convergence even for strongly convex matrix objectives
\cite{DBLP:journals/jmlr/ZhangFZ23}.
For overparameterized positive-semidefinite matrix sensing,
Zhuo et al.~\cite{zhuo2021computational} give a sublinear
convergence guarantee for ordinary FGD up to a statistical error level.
Scaled or preconditioned factor-gradient methods {mitigate these difficulties}
\cite{10.5555/3546258.3546408,DBLP:journals/jmlr/ZhangFZ23}.

{When $r=r_*$, local linear convergence is consistent with standard
smooth-manifold theory for Riemannian gradient descent at a
nondegenerate minimizer; see Boumal~\cite[Theorem~4.20]{boumal2023intromanifolds}.
ProjGD projects a full Euclidean gradient step, whereas Riemannian
gradient descent retracts a tangential step, so the updates need not
coincide. The local result below also covers $r>r_*$, when $\bX_*$
lies on the rank-deficient boundary of $\mathcal M_r$ and is not a
point of the smooth rank-$r$ manifold. The cited manifold theorem
does not directly cover this boundary case.}

\begin{theorem}[Local convergence of ProjGD]\label{thm:local}
Under A1$'$--A2, let $\rank(\bX^{(0)})\leq r$ and
\[
f(\bX^{(0)})-f(\bX_*)\leq 0.01\,\sigma_{r_*}(\bX_*)^2\mu^2/L.
\]
For every $0<\eta<1/L$, ProjGD converges linearly to $\bX_*$, with
\begin{equation}\label{eq:thmlocal}
f(\bX^{(k+1)})-f(\bX_*)\leq
\left[1-\frac{\mu}{75L}(\eta L-\eta^2L^2)\right]
\bigl(f(\bX^{(k)})-f(\bX_*)\bigr).
\end{equation}
\end{theorem}

\section{Landscape consequences and sharpness}\label{sec:landscape}

McRae and Zhang~\cite{mcrae2026matrixlasso} established the sharp
rank-dependent threshold for the absence of spurious local minima
in asymmetric low-rank factorization, covering matrix sensing and
more general convex losses. Under Assumptions~A1--A3, our critical-point
characterization gives the corresponding no-spurious-local-minima
property directly for $f$ over $\mathcal M_r$ at the same threshold.

\subsection{An alternative proof of the sharp threshold}
We next turn the critical-point characterization in Theorem~\ref{thm:gap}
 into a statement about second-order behavior under the rank constraint. At a rank-$r$ critical point, the first-order condition
\[
P_{T(\bX)}\nabla f(\bX)=0
\]
implies that the Euclidean gradient lies entirely in $T(\bX)^\perp$. The following lemma shows that if this normal component is sufficiently large relative to the $r$-th singular value of $\bX$, then there is a rank-preserving curve through $\bX$ along which $f$ has negative second derivative. Consequently, $\bX$ cannot be a local minimizer in $\mathcal{M}_r$.

\begin{lemma}[Necessary condition for a local minimizer]\label{lemma:localminimizer}
Assume A1. If $\rank(\bX)=r$, $P_{T(\bX)}\nabla f(\bX)=0$, and
$\|P_{T(\bX)^\perp}\nabla f(\bX)\|>L\sigma_r(\bX),$ 
then $\bX$ is not a local minimizer of $f$ over matrices of rank $r$. More precisely, there exists a smooth  curve $\bX(t)\in\mathcal M_r$ with $\bX(0)=\bX$ and $\|\dot{\bX}(0)\|_F=1$ such that $
\left.\frac{\mathrm d^2}{\mathrm dt^2}f(\bX(t))\right|_{t=0}
\leq L-\frac{\|P_{T(\bX)^\perp}\nabla f(\bX)\|}{\sigma_r(\bX)}<0.$
\end{lemma}
The proof in Appendix~\ref{sec:proof_localminimizer} constructs this curve on the manifold of rank-$r$ matrices. The same
construction yields the landscape characterization in the following
theorem.

Combining Theorem~\ref{thm:gap} with
Lemma~\ref{lemma:localminimizer} gives the following result for
$f$ over $\mathcal M_r$, including a quantitative bound on the
negative second derivative along a rank-preserving curve.
Rank-deficient boundary points are treated separately.

\begingroup

\begin{theorem}[Sharp asymmetric landscape threshold]\label{thm:landscape}
Under Assumptions A1--A3, $\bX_*$ is the unique local minimizer of $f$ over $\mathcal M_r$. If $L>\mu$, each nonoptimal rank-$r$ critical point $\bX$ is a strict saddle: there exists a smooth rank-preserving curve \(\bX(t)\in\calM_r\), with \(\bX(0)=\bX\) and \(\|\dot \bX(0)\|_F=1\), satisfying
\[
\left.\frac{\mathrm d^2}{\mathrm dt^2}f(\bX(t))\right|_{t=0}
\leq L-\frac{\|P_{T(\bX)^\perp}\nabla f(\bX)\|}{\sigma_r(\bX)}
\leq L-\sqrt{\frac{r}{r_*}}\,\frac{2L\mu}{L-\mu}<0.
\]
\end{theorem}
\endgroup
The proof is given in Appendix~\ref{sec:proof_landscape}.

\paragraph{Interpretation and matrix-sensing specialization}
The threshold for $f$ over $\mathcal M_r$ is
$L/\mu<1+2\sqrt{r/r_*}$, reducing to $L/\mu<3$
when $r=r_*$. Its dependence on $r/r_*$ describes the effect of
rank overparameterization, although the restricted constants
themselves can vary with the tested rank. For
$f(\bX)=\frac12\|\mathcal A(\bX-\bX_*)\|_2^2$, rank-$(r+r_*)$
RIP with constant $\delta$ gives $\mu=1-\delta$ and $L=1+\delta$.
Thus the condition becomes
\[
\delta<\frac{\sqrt{r/r_*}}{1+\sqrt{r/r_*}},
\]
matching the sharp rank-dependent threshold for asymmetric
factorization in \cite{mcrae2026matrixlasso}; the corresponding PSD threshold
was established in \cite{zhang2022improved}.
The theorem additionally gives a quantitative bound on the
negative second derivative along a rank-preserving curve.%

For nonlinear objectives, A1 is a restricted second-derivative condition rather than a quadratic measurement model. The proof uses only rank-$(r+r_*)$ restricted smoothness and strong convexity because the compressed second-derivative operator acts on a linear subspace all of whose elements have rank at most $r+r_*$. Related landscape results under restricted smoothness and strong convexity include \cite{uschmajew2020critical,Zhang2021GeneralLM}; our emphasis is the critical-point characterization and its use in the algorithmic result, with the landscape statement as a consequence.

\subsection{Sharpness of the critical-point characterization and step-size threshold}\label{sec:sharpness}
The following quadratic construction attains both the
normal-component bound in Theorem~\ref{thm:gap} and the lower
step-size threshold in Theorem~\ref{thm:main1}. It parallels known
sharp asymmetric examples
\cite{zhang2022improved,zhang2025noisy,mcrae2026matrixlasso}
and exhibits strict nonoptimal local minima beyond the landscape
threshold.%

Partition the coordinates into blocks of sizes $r_*$, $r$, and $n-r-r_*$ (the third block may be absent), and set
\[
\bX_*=\diag(\bI_{r_*},0,0),\qquad
\tau(\bX)=\frac{\operatorname{tr}(\bX_{11}-\bI_{r_*})}{\sqrt{r_*}}+\frac{\operatorname{tr}(\bX_{22})}{\sqrt r}.
\]
For $\kappa>1$, consider
\begin{equation}\label{eq:sharp_f}
f(\bX)=\frac12\|\bX-\bX_*\|_F^2+\frac{\kappa-1}{4}\left[\tau(\bX)^2+\|\bX_{12}+\bX_{21}^T\|_F^2\right].
\end{equation}

\begin{proposition}[Sharpness of the normal-component bound and step-size threshold]\label{prop:sharp}
Let $1\leq r_*\leq r$ and $n\geq r+r_*$. The function in \eqref{eq:sharp_f} satisfies A1--A2 with $\mu=1$ and $L=\kappa$. Define
\[
s=\sqrt{\frac{r_*}{r}}\frac{\kappa-1}{\kappa+1},\qquad
\bar\bX=\diag(0,s\bI_r,0)\in\mathcal M_r.
\]
Then $P_{T(\bar\bX)}\nabla f(\bar\bX)=0$ and
\[
\frac{\|P_{T(\bar\bX)^\perp}\nabla f(\bar\bX)\|}{\sigma_r(\bar\bX)}
=\sqrt{\frac r{r_*}}\frac{2\kappa}{\kappa-1}.
\]
Moreover, $\bar\bX\ne\bX_*$ is a ProjGD fixed point for every
$0<\eta\leq\sqrt{r_*/r}\,(\kappa-1)/(2\kappa)$.
It is a strict local minimizer when $r<\frac14(\kappa-1)^2r_*$,
a degenerate nonoptimal critical point at equality, and a strict saddle
when $r>\frac14(\kappa-1)^2r_*$.
\end{proposition}

The proof is given in Appendix~\ref{sec:proof_sharp}. 

\section{Connection with Burer--Monteiro factorization}\label{sec:balanced}

Consider the factorized objective
$h(\bU,\bV)=f(\bU\bV^T)$, where
$\bU,\bV\in\reals^{n\times r}$.
The uniqueness of the ProjGD fixed point also yields an optimality
guarantee for second-order stationary points of $h$. The connection is provided by
Ha, Liu, and Barber~\cite[Theorem~1]{doi:10.1137/18M1231675}:
if $(\bU,\bV)$ is a second-order stationary point of $h$, then
$\bX=\bU\bV^T$ is a ProjGD fixed point for every
$0<\eta\leq 1/\beta_{\mathrm{local}}(\bX)$, where
$\beta_{\mathrm{local}}(\bX)$ is the local smoothness constant
over matrices of rank at most $r$.
This correspondence holds for arbitrary search ranks, including
$r>r_*$.
Under Assumption~A1$'$, the local smoothness constant is bounded
by $L$. We can therefore combine this correspondence with
Theorem~\ref{thm:fixedpoint}, which excludes nonoptimal ProjGD
fixed points whenever the step size exceeds
\[
\eta_{\mathrm{crit}}
:=\sqrt{\frac{r_*}{r}}\frac{L-\mu}{2L\mu}.
\]

\begin{corollary}[Global optimality of factored second-order
stationary points]\label{cor:bm_sosp}
Under A1$'$--A3, every second-order stationary point
$(\bU,\bV)$ of $h(\bU,\bV)=f(\bU\bV^T)$ satisfies
$\bU\bV^T=\bX_*$, including when $r>r_*$.
\end{corollary}

\begin{proof}
Assumption~A3 ensures that $\eta_{\mathrm{crit}}<1/L$.
Choose any $\eta\in(\eta_{\mathrm{crit}},1/L)$.
Since A1$'$ implies
$\beta_{\mathrm{local}}(\bU\bV^T)\leq L$,
the correspondence of Ha, Liu, and Barber shows that
$\bU\bV^T$ is a ProjGD fixed point at this step size.
Theorem~\ref{thm:fixedpoint} then gives
$\bU\bV^T=\bX_*$.
\end{proof}

For $r=r_*$, A1 and A1$'$ coincide. After scaling $f$ by
$2/(L+\mu)$, these assumptions imply the $\delta$-RIP$_{2r,2r}$
condition of \cite{Zhang2021GeneralLM}, where
$\delta=(L-\mu)/(L+\mu)$; our curvature bounds additionally
hold at all ambient base points. Assumption~A3 then reads
$\delta<1/2$, so Corollary~\ref{cor:bm_sosp} rules out spurious
second-order critical points of $f(\bU\bV^T)$ throughout this
range under the present assumptions. The threshold is sharp
by \cite[Example~2]{Zhang2021GeneralLM}, which exhibits such a
point at $\delta=1/2$. For comparison,
\cite[Corollary~2]{Zhang2021GeneralLM} gives the general-rank
bound $\delta\leq1/3$ under its restricted-curvature assumptions;
the threshold $\delta<1/2$ also follows from
\cite{mcrae2026matrixlasso}.


\section{Discussion}\label{sec:discussion}
\paragraph{Algorithm-independent role of the critical-point characterization}
The bound \eqref{eq:gap} is a property of the objective and the
geometry of the rank-$r$ manifold; neither an algorithm nor a step
size appears in its statement. Its usefulness is therefore not
restricted to ProjGD. Whenever the limit points or accumulation points
of a method can be related to rank-$r$ critical points, the
characterization constrains the geometry of possible nonoptimal
limits and can be combined with method-specific descent, stability,
or saddle-avoidance arguments. The ProjGD analysis in Sections~\ref{sec:certificate}--\ref{sec:algorithms}
is one realization of this principle, while Section~\ref{sec:balanced}
illustrates how related critical-point information can also be
transferred to a factorized formulation.

\paragraph{{What the critical-point characterization adds to the fixed-point correspondence}}
Ha, Liu, and Barber~\cite{doi:10.1137/18M1231675} relate
second-order stationary points of the factored objective to ProjGD
fixed points. Under the assumptions of
Lemma~\ref{lemma:uniform_progress}, the uniform improvement
estimate \eqref{eq:uniform_step} additionally controls the size of
every ProjGD update relative to the current error at arbitrary
feasible iterates, including near rank-deficient matrices.
Combined with the sufficient-decrease estimate in
Lemma~\ref{lemma:descent2}, it gives the objective contraction
in Theorem~\ref{thm:main1}.

\paragraph{{Rank overparameterization and the step-size window}}
The lower endpoint
$\eta_{\mathrm{crit}}=\sqrt{r_*/r}(L-\mu)/(2L\mu)$ arises from
normal singular-value competition under hard projection, while
$\eta<1/L$ ensures descent under rank-$2r$ restricted smoothness.
For $r=r_*$, the condition becomes $L/\mu<3$; for $r>r_*$,
the permissible restricted condition number grows as
$1+2\sqrt{r/r_*}$. Thus rank overparameterization improves the
worst-case critical-point characterization. The quadratic construction of
Proposition~\ref{prop:sharp} shows why the lower step-size condition
is not simply a proof artifact. It does \emph{not} imply that taking
an arbitrarily large $r$ improves practical performance, because the
restricted smoothness and strong convexity constants themselves depend on the tested rank.

\paragraph{{Matrix geometry versus factor geometry}}
The objective $f(\bX)=\frac12\|\bX-\bX_*\|_F^2$ has
identity ambient Hessian, but the induced factor objective
$f(\bU\bV^T)$ has a Hessian whose eigenvalues may depend on the
nonzero singular values of $\bX_*$. Plain factor-gradient methods can therefore
exhibit poor conditioning or sublinear behavior under rank
overspecification, whereas suitable preconditioned methods can
mitigate this effect
\cite{10.5555/3546258.3546408,DBLP:journals/jmlr/ZhangFZ23}.
ProjGD operates in matrix coordinates: the rate in
Theorem~\ref{thm:main1} depends on $L,\mu,r,r_*,\eta$ but not
on $\sigma_1(\bX_*)/\sigma_{r_*}(\bX_*)$.
{The comparison concerns ordinary, unpreconditioned FGD;
scaled and preconditioned factor methods can also achieve rates
independent of this target condition number under their respective
assumptions.}

\paragraph{{Projection cost}}
{Our convergence bounds count iterations with exact best rank-$r$
projections. A full SVD of a dense $n\times n$ matrix costs $O(n^3)$
arithmetic operations. The cost of a partial SVD depends on matrix-vector
access and accuracy requirements; iterative singular-subspace estimates
may also depend on spectral separation near the truncation threshold.
Consequently, a rate independent of the target condition number
does not establish an equivalent arithmetic cost or a runtime advantage
over factor or Riemannian methods, which can exploit low-rank and
measurement structure
\cite{10.5555/3546258.3546408,DBLP:journals/jmlr/ZhangFZ23,vandereycken2013lowrank}.
An inexact projection changes the update analyzed here and requires
additional error control.}

\paragraph{{Scope and limitations}}
The deterministic global result assumes exact best rank-$r$ projections,
a low-rank unconstrained minimizer, rank-$(r+r_*)$ restricted smoothness and strong convexity,
and rank-$2r$ restricted smoothness. The rank-$(r+r_*)$ restricted smoothness and strong convexity
conditions suffice for the critical-point characterization and landscape theorem, while
the stronger smoothness range is used for sufficient decrease in ProjGD.
{For least-squares matrix completion, any unobserved entry $(i,j)$
gives a rank-one direction $\bE=\mathbf e_i\mathbf e_j^T$ with
$\langle\nabla^2 f(\bX)[\bE],\bE\rangle=0$.
Thus the restricted strong convexity bound in A1 fails for every $\mu>0$.
Matrix-completion guarantees require structural or local assumptions
outside the present hypotheses; the global upper smoothness bound
of the sampling loss does not resolve this failure.}
Noise, approximate projections, nonzero $\nabla f(\bX_*)$,
and approximate second-order stationarity are outside the present
theorem; some of these aspects are addressed in
\cite{zhang2025noisy,mcrae2026matrixlasso} by different methods.
Possible extensions include adaptive step sizes, inexact truncated
SVDs, and safeguarded switches between a low-cost Riemannian update
and a ProjGD update. Any such hybrid strategy requires its own
analysis at rank-deficient boundaries
\cite{levin2023finding,olikier2023apocalypse}.

\section{Conclusions}

We establish three connected results for rank-constrained
projected gradient descent.
First, combining the compressed centered-gradient estimate in Lemma~\ref{lemma:gradient} with
the normal-perturbation bound in Lemma~\ref{lemma:orthogonalrate0} gives a sharp quantitative
critical-point characterization: at every nonoptimal rank-$r$ critical point,
the normal component of the gradient has an explicit rank-dependent
lower bound relative to $\sigma_r(\bX)$.
Second, this characterization excludes nonoptimal ProjGD fixed points
whenever $\eta>\eta_{\mathrm{crit}}$ under A1--A2.
The quadratic example in Proposition~\ref{prop:sharp} attains
the critical-point bound and the lower step-size threshold.
Third, Lemma~\ref{lemma:uniform_progress} extends the argument to a
uniform improvement estimate at arbitrary feasible iterates.

Under the additional rank-$2r$ restricted smoothness assumption,
this progress estimate and sufficient decrease in Lemma~\ref{lemma:descent2} yield deterministic
global linear convergence for
\[
\eta_{\mathrm{crit}}<\eta<\frac1L.
\]
The interval is nonempty when $L/\mu<1+2\sqrt{r/r_*}$,
and the rate does not depend on the singular-value condition number
of $\bX_*$, even when $r>r_*$.
The proof connects a sharp description of nonoptimal critical
points to quantitative progress at every iterate,
using the same matrix inequality and its direct Eckart--Young
argument throughout. As a secondary consequence, the critical-point characterization
also gives the no-spurious-local-minima property for $f$ over
$\mathcal M_r$ at the known sharp asymmetric
threshold~\cite{mcrae2026matrixlasso} and gives an explicit negative
second-derivative bound along a rank-preserving curve at each
nonoptimal rank-$r$ critical point under the stated assumptions.

\appendix

\section{ProjGD estimates and complementary local convergence}\label{sec:local_appendix}\label{sec:proof}
This appendix contains the projection estimates used by the global ProjGD theorem and a complementary local convergence result that does not require Assumption~A3. The estimates are valid at iterates of their actual rank and do not require the Euclidean gradient to have low rank.

\subsection{Projection step estimates}
\begin{proof}[Proof of Lemma~\ref{lemma:descent2}]
Let $\bY=\bX-\eta\nabla f(\bX)$ and $\bB=\Pr(\bY)$. Optimality of $\bB$ and the condition $\rank(\bX)\leq r$ give
\[
\|\bB-\bY\|_F^2\leq\|\bX-\bY\|_F^2,
\qquad
\langle\nabla f(\bX),\bB-\bX\rangle\leq-\frac1{2\eta}\|\bB-\bX\|_F^2.
\]
Since $\rank(\bB-\bX)\leq2r$, Assumption~A1$'$ applies to the line segment from $\bX$ to $\bB$ and gives
\[
f(\bB)\leq f(\bX)+\langle\nabla f(\bX),\bB-\bX\rangle+\frac L2\|\bB-\bX\|_F^2.
\]
Combining this with the preceding projection inequality yields \eqref{eq:descent2}.

Write $\bR=\bY-\bB$ and let $\bP,\bQ$ project onto the column and row spaces of $\bX$, whose rank is $q\leq r$. If $\bR=0$, the tangent estimate is immediate. Otherwise $\rank(\bB)=r$; let $\bP_B,\bQ_B$ be its column and row projectors, and $s=\sigma_r(\bB)$. The truncated SVD gives $\bP_B\bR=0$, $\bR\bQ_B=0$, and $\|\bR\|\leq s$. Consequently,
\begin{align*}
\|\bP\bR\|_F
&\leq s\|\bP(\bI-\bP_B)\|_F
\leq s\|(\bI-\bP)\bP_B\|_F
\leq\|(\bI-\bP)\bB\|_F
\leq\|\bB-\bX\|_F.
\end{align*}
The middle projector inequality follows from
$q-\operatorname{tr}(\bP\bP_B)\leq r-\operatorname{tr}(\bP\bP_B)$.
Similarly, $\|\bR\bQ\|_F\leq\|\bB-\bX\|_F$. The orthogonal decomposition
$P_{T(\bX)}\bR=\bP\bR+(\bI-\bP)\bR\bQ$ therefore gives
$\|P_{T(\bX)}\bR\|_F\leq\sqrt2\|\bB-\bX\|_F$. Hence
\[
\|P_{T(\bX)}(\bY-\bX)\|_F
\leq(1+\sqrt2)\|\bB-\bX\|_F
\leq\sqrt{10}\|\bB-\bX\|_F.
\]
This proves the tangent part of \eqref{eq:descent22}, also when $q=0$.

For the spectral part, the best-approximation property always gives
$\|\bY-\bB\|\leq\sigma_r(\bB)$ (both sides are zero if $\rank(\bB)<r$).
Weyl's inequality then yields
\[
\eta\|\nabla f(\bX)\|\leq\|\bY-\bB\|+\|\bB-\bX\|
\leq\sigma_r(\bX)+2\|\bB-\bX\|.
\]
Combining this with $\|\bB-\bX\|\leq\|\bB-\bX\|_F$ proves the other part of \eqref{eq:descent22}.
\end{proof}

\subsection{Normal error near a low-rank target}\label{sec:Auxiliary}
\begin{lemma}[Normal error estimate]\label{lemma:local}
Let $\rank(\bX_*)=r_*>0$, let $\bE=\bX-\bX_*$, and set $s_*=\sigma_{r_*}(\bX_*)$. Then
\[
\|P_{T(\bX)^\perp}\bE\|_F\leq\frac{\|\bE\|_F^2}{s_*}.
\]
In particular, if $0<\|\bE\|_F<c_0s_*$ with $c_0<1$, the angle $\theta$ between $\bE$ and $T(\bX)$ satisfies $\sin\theta\leq\|\bE\|_F/s_*<c_0$.
\end{lemma}
\begin{proof}
Let $\bP,\bQ$ be the column and row projectors of $\bX$. Since $(\bI-\bP)\bX=0$ and $\bX(\bI-\bQ)=0$,
\[
(\bI-\bP)\bX_*(\bI-\bQ)
=(\bI-\bP)\bE\bX_*^\dagger\bE(\bI-\bQ).
\]
The normal component of $\bE$ is the negative of the left side. Taking norms and using $\|\bX_*^\dagger\|=1/s_*$ proves the estimate. The angle statement follows by orthogonal projection.
\end{proof}

\subsection{Local convergence Theorem~\ref{thm:local}}

\begin{proof}[Proof of Theorem~\ref{thm:local}]
Write $\ell(\bX)=f(\bX)-f(\bX_*)$, $\bG=\nabla f(\bX)$, $\bE=\bX-\bX_*$, and $\widetilde{\bG}=P_{\mathcal S_{\bX}}\bG$. Since $\rank(\bE)\leq r+r_*$, A1 and $\nabla f(\bX_*)=0$ imply
\begin{equation}\label{eq:proof_local10}
\frac\mu2\|\bE\|_F^2\leq\ell(\bX)\leq\frac L2\|\bE\|_F^2.
\end{equation}
The restriction of $f$ to the affine subspace $\bX_*+\mathcal S_{\bX}$ is $\mu$-strongly convex and $L$-smooth. Since $\bX_*$ minimizes this restriction, the usual smooth-convex inequalities applied inside $\mathcal S_{\bX}$ give, for $\bX\neq\bX_*$ and $\bu=\bE/\|\bE\|_F$,
\begin{equation}\label{eq:proof_local1}
\|\widetilde{\bG}\|_F\leq\sqrt{2L\ell(\bX)},\qquad
\langle\widetilde{\bG},\bu\rangle\geq\sqrt{2\mu\ell(\bX)}.
\end{equation}
For completeness, the first inequality follows by taking one $L$-smooth step $\bX-L^{-1}\widetilde{\bG}$ within the affine subspace and comparing with the global minimizer $\bX_*$; the second follows from restricted strong convexity along $\bE$ as in the standard proof.

It suffices to work on the slightly larger sublevel set
\begin{equation}\label{eq:local_proof0}
\ell(\bX)\leq\frac{s_*^2\mu^2}{72L},\qquad s_*=\sigma_{r_*}(\bX_*).
\end{equation}
By \eqref{eq:proof_local10} and Lemma~\ref{lemma:local}, the angle $\theta$ between $\bE$ and $T(\bX)$ satisfies
\[
\sin\theta\leq\frac{\|\bE\|_F}{s_*}\leq\frac16\sqrt{\frac\mu L},\qquad
\cos\theta\geq\frac56.
\]
Let $\bv=P_{T(\bX)}\bu/\|P_{T(\bX)}\bu\|_F$. Because $\bu\in\mathcal S_{\bX}$ and $P_{T(\bX)}$ commutes with $P_{\mathcal S_{\bX}}$, also $\bv\in\mathcal S_{\bX}$. Hence $\langle\bG,\bv\rangle=\langle\widetilde{\bG},\bv\rangle$. Resolving $\bv$ into its component along $\bu$ and its orthogonal component, \eqref{eq:proof_local1} yields
\[
\|P_{T(\bX)}\bG\|_F\geq\langle\bG,\bv\rangle
\geq\sqrt{2\mu\ell(\bX)}\cos\theta
-\sqrt{2L\ell(\bX)}\sin\theta
\geq\frac23\sqrt{2\mu\ell(\bX)}.
\]
Applying Lemma~\ref{lemma:descent2} gives
\begin{equation}\label{eq:local_proof}
\ell(\bX)-\ell(\bX_+)
\geq\frac{\eta^2}{20}\left(\frac1\eta-L\right)\|P_{T(\bX)}\bG\|_F^2
\geq\frac\mu{75}\left(\frac1\eta-L\right)\eta^2\ell(\bX).
\end{equation}
The last constant is deliberately weakened. Descent preserves \eqref{eq:local_proof0}, and the stated initialization bound is smaller because $0.01<1/72$. Induction proves \eqref{eq:thmlocal}. If an iterate equals $\bX_*$, all subsequent iterates do as well. Finally, \eqref{eq:proof_local10} converts convergence of the objective to convergence of the matrices.
\end{proof}

\section{{Proof of the compressed centered-gradient estimate}}\label{sec:proof_gradient}
{This appendix proves Lemma~\ref{lemma:gradient}, used in Sections~\ref{sec:certificate} and~\ref{sec:algorithms}.}
\begin{proof}
Let $\bE=\bX-\bX_*$ and abbreviate $\mathcal S=\mathcal S_{\bX}$. Because $\bE\in\mathcal S$ and every element of $\mathcal S$ has rank at most $r+r_*$, Assumption~A1 implies that, for every $t\in[0,1]$, the compressed self-adjoint operator
\[
\mathcal H_t=P_{\mathcal S}\nabla^2f(\bX_*+t\bE)\big|_{\mathcal S}
\]
satisfies $\mu\mathcal I_{\mathcal S}\preceq\mathcal H_t\preceq L\mathcal I_{\mathcal S}$. Hence the averaged compressed second-derivative operator $\overline{\mathcal H}=\int_0^1\mathcal H_t\,dt$ satisfies
\[
\left\|\overline{\mathcal H}-\frac{L+\mu}{2}\mathcal I_{\mathcal S}\right\|_{F\to F}
\leq\frac{L-\mu}{2}.
\]
The fundamental theorem of calculus, followed by orthogonal projection onto $\mathcal S$, gives
\[
P_{\mathcal S}\bigl(\nabla f(\bX)-\nabla f(\bX_*)\bigr)=\overline{\mathcal H}[\bE].
\]
Since $\nabla f(\bX_*)=0$, applying the operator bound to $\bE$ proves \eqref{eq:compressed_gradient}.
\end{proof}

\section{Detailed proof of Proposition~\ref{prop:sharp}}\label{sec:proof_sharp}
\begin{proof}
Set $\bW=\diag(\bI_{r_*}/\sqrt{r_*},\bI_r/\sqrt r,0)$ and
$\bZ_{ij}=\mathbf e_i\mathbf e_j^T+\mathbf e_j\mathbf e_i^T$
for $1\leq i\leq r_*<j\leq r_*+r$. Then
$\tau(\bX)=\langle\bW,\bX-\bX_*\rangle$ and
$\|\bX_{12}+\bX_{21}^T\|_F^2=\sum_{i,j}\langle\bZ_{ij},\bX-\bX_*\rangle^2$.
The matrices $\bW$ and $\bZ_{ij}$ are mutually orthogonal, each with squared Frobenius norm $2$. Hence
\[
\nabla^2 f=\mathcal I+\frac{\kappa-1}{2}
\left(\bW\otimes\bW+\sum_{i,j}\bZ_{ij}\otimes\bZ_{ij}\right),
\]
where $(\bA\otimes\bA)[\bE]=\langle\bA,\bE\rangle\bA$.
Its eigenvalues are $1$ and $\kappa$, proving A1 with $\mu=1$ and $L=\kappa$. Since $f\geq0=f(\bX_*)$, A2 also holds.

At $\bar\bX$, $\tau(\bar\bX)=-2\sqrt{r_*}/(\kappa+1)$ and
\[
\nabla f(\bar\bX)=-g\,\diag(\bI_{r_*},0,0),\qquad
g=\frac{2\kappa}{\kappa+1}.
\]
This gradient is normal to $T(\bar\bX)$, with spectral norm $g$, whereas $\sigma_r(\bar\bX)=s$. Thus the stated ratio equals $g/s$. Moreover,
$\bar\bX-\eta\nabla f(\bar\bX)=\diag(\eta g\bI_{r_*},s\bI_r,0)$,
whose best rank-$r$ approximations include $\bar\bX$ whenever $\eta g\leq s$, establishing the fixed-point claim, including the case of tied singular values.

For the local landscape, let $K$ denote the union of the first and third coordinate blocks. Every $\bX\in\mathcal M_r$ sufficiently near $\bar\bX$ has an invertible $(2,2)$ block, and rank $r$ then forces its Schur complement to vanish. Using local coordinates
$\bA=\bX_{K2}$, $\bB=\bX_{2K}$, and $\bC=\bX_{22}-s\bI_r$, we have
\[
\bX_{KK}=\bA(s\bI_r+\bC)^{-1}\bB
=\frac1s\bA\bB+O(\|(\bA,\bB,\bC)\|^3).
\]
Write $a_{ij}=(\bX_{12})_{i,j-r_*}$ and $b_{ji}=(\bX_{21})_{j-r_*,i}$. Since $\nabla f(\bar\bX)=-g\,\diag(\bI_{r_*},0,0)$, the linear term in the Taylor expansion of $f$ contributes
$-(g/s)\operatorname{tr}(\bX_{12}\bX_{21})$ at quadratic order. Together with the Hessian terms, this yields
\begin{multline*}
f(\bX)-f(\bar\bX)
=\frac12\sum_{i\leq r_*<j\leq r_*+r}
\begin{pmatrix}a_{ij}\\ b_{ji}\end{pmatrix}^{\!T}
\left(\bB_\kappa-\frac gs\bJ\right)
\begin{pmatrix}a_{ij}\\ b_{ji}\end{pmatrix}\\
+\frac12\left(\|\bX_{32}\|_F^2+\|\bX_{23}\|_F^2+\|\bC\|_F^2\right)
+\frac{\kappa-1}{4r}\operatorname{tr}(\bC)^2
+O(\|(\bA,\bB,\bC)\|^3),
\end{multline*}
where
\[
\bB_\kappa=\frac12\begin{pmatrix}\kappa+1&\kappa-1\\ \kappa-1&\kappa+1\end{pmatrix},
\qquad \bJ=\begin{pmatrix}0&1\\1&0\end{pmatrix}.
\]
Each $2\times2$ block has eigenvalues $\kappa-g/s$ and $1+g/s$, while the remaining quadratic terms are positive definite. Since
\[
\kappa-\frac gs
=\frac{\kappa}{\kappa-1}\left(\kappa-1-2\sqrt{\frac r{r_*}}\right),
\]
the constrained Hessian is positive definite when $r<\frac14(\kappa-1)^2r_*$, has a zero eigenvalue at equality, and has a negative eigenvalue when $r>\frac14(\kappa-1)^2r_*$. This gives the respective local-minimum, degeneracy, and strict-saddle conclusions.
\end{proof}

\section{Proof of Lemma~\ref{lemma:localminimizer}}\label{sec:proof_localminimizer}
\begin{proof}
Write $s=\sigma_r(\bX)>0$ and $\bG_N=P_{T(\bX)^\perp}\nabla f(\bX)$. Let $\bu,\bv$ be a singular pair of $\bX$ for $s$, and choose unit singular vectors $\bp,\bq$ of $\bG_N$ so that $\bp^T\bG_N\bq=\|\bG_N\|$. They are orthogonal to the column and row spaces of $\bX$. Define
\[
\widetilde{\bX}(t)=\bX+t(\bu\bq^T-\bp\bv^T)-\frac{t^2}{s}\bp\bq^T.
\]
Replacing the term $s\bu\bv^T$ in the SVD of $\bX$ by
$(s\bu-t\bp)(\bv+(t/s)\bq)^T$ shows that $\rank(\widetilde{\bX}(t))=r$ for all sufficiently small $|t|$. Set $\bDelta(t)=\widetilde{\bX}(t)-\bX$. The first-order condition $P_{T(\bX)}\nabla f(\bX)=0$ and orthogonality give
\[
\langle\nabla f(\bX),\bDelta(t)\rangle=-\frac{\|\bG_N\|}{s}t^2,
\qquad
\|\bDelta(t)\|_F^2=2t^2+\frac{t^4}{s^2}.
\]
Since $\rank(\bDelta(t))\leq2\leq r+r_*$, Assumption~A1 applies to the line segment in direction $\bDelta(t)$, and therefore
\[
f(\widetilde{\bX}(t))-f(\bX)
\leq\left(L-\frac{\|\bG_N\|}{s}\right)t^2+\frac L{2s^2}t^4.
\]
Now reparameterize by
\[
\bX(t)=\widetilde{\bX}(t/\sqrt2).
\]
Then $\bX(t)$ is smooth and rank preserving for all sufficiently small $|t|$, $\bX(0)=\bX$, and
\[
\|\dot{\bX}(0)\|_F
=\frac1{\sqrt2}\|\bu\bq^T-\bp\bv^T\|_F=1.
\]
Moreover,
\[
f(\bX(t))-f(\bX)
\leq \frac12\left(L-\frac{\|\bG_N\|}{s}\right)t^2+\frac{L}{8s^2}t^4.
\]
Because the first derivative at $t=0$ vanishes, dividing by $t^2/2$ and letting $t\to0$ gives
\[
\left.\frac{\mathrm d^2}{\mathrm dt^2}f(\bX(t))\right|_{t=0}
\leq L-\frac{\|\bG_N\|}{s}<0.
\]
Hence $\bX$ is not a local minimizer of $f$ over matrices of rank $r$, and the quantitative claim in the lemma follows.
\end{proof}

\section{Proof of Theorem~\ref{thm:landscape}}\label{sec:proof_landscape}
\begingroup

\begin{proof}
First suppose $0<\mu<L$ and let $\bX\neq\bX_*$ be a rank-$r$ critical point. Theorem~\ref{thm:gap} gives
\[
\frac{\|P_{T(\bX)^\perp}\nabla f(\bX)\|}{\sigma_r(\bX)}
\geq\sqrt{\frac r{r_*}}\frac{2L\mu}{L-\mu}>L.
\]
The proof of Lemma~\ref{lemma:localminimizer} constructs a smooth rank-preserving curve $\widetilde{\bX}(t)$ with $\widetilde{\bX}(0)=\bX$, $\|\dot{\widetilde{\bX}}(0)\|_F^2=2$, and
\[
f(\widetilde{\bX}(t))-f(\bX)\leq
\left(L-\frac{\|P_{T(\bX)^\perp}\nabla f(\bX)\|}{\sigma_r(\bX)}\right)t^2+O(t^4).
\]
Reparameterizing by $\bX(t)=\widetilde{\bX}(t/\sqrt2)$ gives $\|\dot{\bX}(0)\|_F=1$ and the displayed second-derivative bound. Thus $\bX$ is a strict saddle along a rank-preserving curve.
If $L=\mu$, the argument in the proof of
Theorem~\ref{thm:fixedpoint} rules out nonoptimal
rank-$r$ critical points altogether.
Finally, let $\rank(\bX)=q<r$ and $\bX\ne\bX_*$. By
\eqref{eq:fixedpoint_nonzero_gradient}, $\nabla f(\bX)\ne0$.
Let $\bu,\bv$ be top singular vectors of $\nabla f(\bX)$, with
$\bu^T\nabla f(\bX)\bv=\|\nabla f(\bX)\|>0$.
The curve $\bX(t)=\bX-t\bu\bv^T$ has rank at most
$q+1\leq r$ and strictly decreases $f$ for small $t>0$.
Thus no nonoptimal rank-deficient point in $\mathcal M_r$ is a local minimizer,
and the theorem follows.
\end{proof}
\endgroup

\section{Relative-concavity comparison under A2}\label{sec:rc_comparison}

We derive a baseline convergence bound by combining the hard-projection
relative-concavity inequality of Liu and Barber
\cite[Sections~4--5]{LiuBarber2020} with A2.
Assume A1$'$--A2, set $a=\sqrt{r_*/r}$, and choose
$a/(2\mu)<\eta\leq1/L$.
For any $\bX\in\mathcal M_r$ and any best rank-$r$ approximation
$\bX^+=\Pr(\bX-\eta\nabla f(\bX))$, write
$\bE=\bX-\bX_*$, $\bE^+=\bX^+-\bX_*$,
$\bDelta=\bX^+-\bX$, $\bG=\nabla f(\bX)$, and
$\ell^+=f(\bX^+)-f(\bX_*)$.
Relative concavity with comparator $\bX_*$ gives the first inequality
below. Restricted strong convexity from $\bX$ to $\bX_*$ and restricted
smoothness from $\bX$ to $\bX^+$ then give the second:
\[
\begin{aligned}
2\eta\langle\bG,\bE^+\rangle
&\leq\|\bE\|_F^2-(1-a)\|\bE^+\|_F^2-\|\bDelta\|_F^2,\\
2\eta\ell^+
&\leq(1-\eta\mu)\|\bE\|_F^2-(1-a)\|\bE^+\|_F^2
      -(1-\eta L)\|\bDelta\|_F^2.
\end{aligned}
\]
The lower curvature bound is used only along directions of rank at
most $r+r_*$, and the upper bound along directions of rank at most
$2r$. These are precisely the bounds supplied by A1$'$.
Since A2 and restricted strong convexity give
$\ell^+\geq(\mu/2)\|\bE^+\|_F^2$, and $\eta\leq1/L$, we obtain
\[
\|\bE^+\|_F^2\leq q\|\bE\|_F^2,
\qquad
q:=\frac{1-\eta\mu}{1-a+\eta\mu}\in[0,1).
\]
Since $0<a\leq1$ and $0<\eta\mu\leq1$, the denominator is positive
and $q\geq0$; $q<1$ follows from $2\eta\mu>a$.
Thus the Frobenius distance contracts geometrically from every
initialization of rank at most $r$. The argument covers rank-deficient iterates and every
choice at SVD ties; zero-padding extends the relative-concavity upper
bound to dimensions below $r+r_*$. The step-size interval is nonempty
when $r>\tfrac14(L/\mu)^2r_*$.

\bibliographystyle{spmpsci}
\bibliography{citations2}

\begin{thebibliography}{10}
\providecommand{\url}[1]{{#1}}
\providecommand{\urlprefix}{URL }
\expandafter\ifx\csname urlstyle\endcsname\relax
  \providecommand{\doi}[1]{DOI~\discretionary{}{}{}#1}\else
  \providecommand{\doi}{DOI~\discretionary{}{}{}\begingroup
  \urlstyle{rm}\Url}\fi

\bibitem{Absil2015}
Absil, P.A., Oseledets, I.V.: Low-rank retractions: a survey and new results.
\newblock Comput. Optim. Appl. \textbf{62}(1), 5--29 (2015).
\newblock \doi{10.1007/s10589-014-9714-4}

\bibitem{NIPS2016_b139e104}
Bhojanapalli, S., Neyshabur, B., Srebro, N.: Global optimality of local search
  for low rank matrix recovery.
\newblock In: D.~Lee, M.~Sugiyama, U.~Luxburg, I.~Guyon, R.~Garnett (eds.)
  Advances in Neural Information Processing Systems, vol.~29. Curran
  Associates, Inc. (2016).
\newblock
  \urlprefix\url{https://proceedings.neurips.cc/paper_files/paper/2016/file/b139e104214a08ae3f2ebcce149cdf6e-Paper.pdf}

\bibitem{pmlr-v130-bi21a}
Bi, Y., Lavaei, J.: On the absence of spurious local minima in nonlinear
  low-rank matrix recovery problems.
\newblock In: A.~Banerjee, K.~Fukumizu (eds.) Proceedings of The 24th
  International Conference on Artificial Intelligence and Statistics,
  \emph{Proceedings of Machine Learning Research}, vol. 130, pp. 379--387. PMLR
  (2021).
\newblock \urlprefix\url{https://proceedings.mlr.press/v130/bi21a.html}

\bibitem{boumal2023intromanifolds}
Boumal, N.: An introduction to optimization on smooth manifolds.
\newblock Cambridge University Press (2023).
\newblock \doi{10.1017/9781009166164}.
\newblock \urlprefix\url{https://www.nicolasboumal.net/book}

\bibitem{10.5555/3157382.3157407}
Boumal, N., Voroninski, V., Bandeira, A.S.: The non-convex {B}urer--{M}onteiro
  approach works on smooth semidefinite programs.
\newblock In: Proceedings of the 30th International Conference on Neural
  Information Processing Systems, NIPS'16, pp. 2765--2773. Curran Associates
  Inc., Red Hook, NY, USA (2016)

\bibitem{Boumal2018}
Boumal, N., Voroninski, V., Bandeira, A.S.: Deterministic guarantees for
  {B}urer--{M}onteiro factorizations of smooth semidefinite programs.
\newblock Comm. Pure Appl. Math. \textbf{73}(3), 581--608 (2020).
\newblock \doi{10.1002/cpa.21830}

\bibitem{Burer:2003ve}
Burer, S., Monteiro, R.D.C.: A nonlinear programming algorithm for solving
  semidefinite programs via low-rank factorization.
\newblock Math. Program. \textbf{95}(2), 329--357 (2003).
\newblock \doi{10.1007/s10107-002-0352-8}

\bibitem{doi:10.1137/17M1141394}
Cai, J.F., Wang, T., Wei, K.: Spectral compressed sensing via projected
  gradient descent.
\newblock SIAM J. Optim. \textbf{28}(3), 2625--2653 (2018).
\newblock \doi{10.1137/17M1141394}

\bibitem{chen2015fast}
Chen, Y., Wainwright, M.J.: Fast low-rank estimation by projected gradient
  descent: General statistical and algorithmic guarantees (2015)

\bibitem{eckart1936approximation}
Eckart, C., Young, G.: The approximation of one matrix by another of lower
  rank.
\newblock Psychometrika \textbf{1}(3), 211--218 (1936).
\newblock \doi{10.1007/BF02288367}

\bibitem{pmlr-v70-ge17a}
Ge, R., Jin, C., Zheng, Y.: No spurious local minima in nonconvex low rank
  problems: A unified geometric analysis.
\newblock In: D.~Precup, Y.W. Teh (eds.) Proceedings of the 34th International
  Conference on Machine Learning, \emph{Proceedings of Machine Learning
  Research}, vol.~70, pp. 1233--1242. PMLR (2017).
\newblock \urlprefix\url{https://proceedings.mlr.press/v70/ge17a.html}

\bibitem{doi:10.1137/18M1231675}
Ha, W., Liu, H., Barber, R.F.: An equivalence between critical points for rank
  constraints versus low-rank factorizations.
\newblock SIAM J. Optim. \textbf{30}(4), 2927--2955 (2020).
\newblock \doi{10.1137/18M1231675}

\bibitem{NIPS2010_08d98638}
Jain, P., Meka, R., Dhillon, I.: Guaranteed rank minimization via singular
  value projection.
\newblock In: J.~Lafferty, C.~Williams, J.~Shawe-Taylor, R.~Zemel, A.~Culotta
  (eds.) Advances in Neural Information Processing Systems, vol.~23. Curran
  Associates, Inc. (2010).
\newblock
  \urlprefix\url{https://proceedings.neurips.cc/paper_files/paper/2010/file/08d98638c6fcd194a4b1e6992063e944-Paper.pdf}

\bibitem{JainTewariKar2014}
Jain, P., Tewari, A., Kar, P.: On iterative hard thresholding methods for
  high-dimensional {M}-estimation.
\newblock In: Advances in Neural Information Processing Systems, vol.~27, pp.
  685--693. Curran Associates, Inc. (2014).
\newblock
  \urlprefix\url{https://proceedings.neurips.cc/paper_files/paper/2014/hash/e460882d11d622bd04a5af4b66d1ffd3-Abstract.html}

\bibitem{levin2023finding}
Levin, E., Kileel, J., Boumal, N.: Finding stationary points on bounded-rank
  matrices: a geometric hurdle and a smooth remedy.
\newblock Math. Program. \textbf{199}(1), 831--864 (2023).
\newblock \doi{10.1007/s10107-022-01851-2}

\bibitem{LiuBarber2020}
Liu, H., Barber, R.F.: Between hard and soft thresholding: optimal iterative
  thresholding algorithms.
\newblock Information and Inference: A Journal of the IMA \textbf{9}(4),
  899--933 (2020).
\newblock \doi{10.1093/imaiai/iaz027}

\bibitem{mcrae2026matrixlasso}
McRae, A.D., Zhang, R.Y.: Sharp recovery and landscape guarantees for the
  nonconvex matrix {LASSO} (2026).
\newblock \doi{10.48550/arXiv.2604.19381}.
\newblock ArXiv preprint arXiv:2604.19381

\bibitem{olikier2023apocalypse}
Olikier, G., Absil, P.A.: An apocalypse-free first-order low-rank optimization
  algorithm with at most one rank reduction attempt per iteration.
\newblock SIAM J. Matrix Anal. Appl. \textbf{44}(3), 1421–1435 (2023).
\newblock \doi{10.1137/22M1518256}

\bibitem{olikier2025projected}
Olikier, G., Waldspurger, I.: Projected gradient descent accumulates at
  {B}ouligand stationary points.
\newblock SIAM J. Optim. \textbf{35}(2), 1004–1029 (2025).
\newblock \doi{10.1137/24M1692782}

\bibitem{pmlr-v54-park17a}
Park, D., Kyrillidis, A., Caramanis, C., Sanghavi, S.: {Non-square matrix
  sensing without spurious local minima via the Burer-Monteiro approach}.
\newblock In: A.~Singh, J.~Zhu (eds.) Proceedings of the 20th International
  Conference on Artificial Intelligence and Statistics, \emph{Proceedings of
  Machine Learning Research}, vol.~54, pp. 65--74. PMLR (2017).
\newblock \urlprefix\url{https://proceedings.mlr.press/v54/park17a.html}

\bibitem{TannerWei2013}
Tanner, J., Wei, K.: Normalized iterative hard thresholding for matrix
  completion.
\newblock SIAM J. Sci. Comput. \textbf{35}(5), S104--S125 (2013).
\newblock \doi{10.1137/120876459}

\bibitem{10.5555/3546258.3546408}
Tong, T., Ma, C., Chi, Y.: Accelerating ill-conditioned low-rank matrix
  estimation via scaled gradient descent.
\newblock J. Mach. Learn. Res. \textbf{22}(150), 1--63 (2021).
\newblock \urlprefix\url{https://jmlr.org/papers/v22/20-1067.html}

\bibitem{uschmajew2020critical}
Uschmajew, A., Vandereycken, B.: On critical points of quadratic low-rank
  matrix optimization problems.
\newblock IMA J. Numer. Anal. \textbf{40}(4), 2626--2651 (2020).
\newblock \doi{10.1093/imanum/drz061}

\bibitem{vandereycken2013lowrank}
Vandereycken, B.: Low-rank matrix completion by {R}iemannian optimization.
\newblock SIAM J. Optim. \textbf{23}(2), 1214--1236 (2013).
\newblock \doi{10.1137/110845768}

\bibitem{DBLP:journals/jmlr/ZhangFZ23}
Zhang, G., Fattahi, S., Zhang, R.Y.: Preconditioned gradient descent for
  overparameterized nonconvex {B}urer--{M}onteiro factorization with global
  optimality certification.
\newblock J. Mach. Learn. Res. \textbf{24}(163), 1--55 (2023).
\newblock \urlprefix\url{http://jmlr.org/papers/v24/22-0882.html}

\bibitem{Zhang2021GeneralLM}
Zhang, H., Bi, Y., Lavaei, J.: General low-rank matrix optimization: Geometric
  analysis and sharper bounds.
\newblock In: Advances in Neural Information Processing Systems, vol.~34, pp.
  27369--27380 (2021).
\newblock \urlprefix\url{https://api.semanticscholar.org/CorpusID:233324260}

\bibitem{NEURIPS2018_f8da71e5}
Zhang, R., Josz, C., Sojoudi, S., Lavaei, J.: How much restricted isometry is
  needed in nonconvex matrix recovery?
\newblock In: S.~Bengio, H.~Wallach, H.~Larochelle, K.~Grauman,
  N.~Cesa-Bianchi, R.~Garnett (eds.) Advances in Neural Information Processing
  Systems, vol.~31. Curran Associates, Inc. (2018).
\newblock
  \urlprefix\url{https://proceedings.neurips.cc/paper_files/paper/2018/file/f8da71e562ff44a2bc7edf3578c593da-Paper.pdf}

\bibitem{zhang2022improved}
Zhang, R.Y.: Improved global guarantees for the nonconvex {Burer--Monteiro}
  factorization via rank overparameterization.
\newblock Math. Program. \textbf{213}(1--2), 1009--1038 (2025).
\newblock \doi{10.1007/s10107-024-02160-6}

\bibitem{zhang2025noisy}
Zhang, R.Y.: Sharp global guarantees for nonconvex low-rank recovery in the
  noisy overparameterized regime.
\newblock SIAM J. Optim. \textbf{35}(3), 2128--2154 (2025).
\newblock \doi{10.1137/24M1697980}

\bibitem{JMLR:v20:19-020}
Zhang, R.Y., Sojoudi, S., Lavaei, J.: Sharp restricted isometry bounds for the
  inexistence of spurious local minima in nonconvex matrix recovery.
\newblock J. Mach. Learn. Res. \textbf{20}(114), 1--34 (2019).
\newblock \urlprefix\url{http://jmlr.org/papers/v20/19-020.html}

\bibitem{zheng2016convergence}
Zheng, Q., Lafferty, J.: Convergence analysis for rectangular matrix completion
  using {B}urer--{M}onteiro factorization and gradient descent (2016)

\bibitem{Zhu2018}
Zhu, Z., Li, Q., Tang, G., Wakin, M.B.: Global optimality in low-rank matrix
  optimization.
\newblock IEEE Trans. Signal Process. \textbf{66}(13), 3614--3628 (2018).
\newblock \doi{10.1109/TSP.2018.2835403}

\bibitem{zhuo2021computational}
Zhuo, J., Kwon, J., Ho, N., Caramanis, C.: On the computational and statistical
  complexity of over-parameterized matrix sensing.
\newblock J. Mach. Learn. Res. \textbf{25}(169), 1--47 (2024).
\newblock \urlprefix\url{https://jmlr.org/papers/v25/21-1437.html}

\end{thebibliography}

\end{document}